\documentclass[11pt,a4paper,leqno]{article}
\usepackage[DIV=11]{typearea}
\usepackage[english]{babel}
\usepackage[utf8]{inputenc}
\usepackage{hyperref}
\usepackage{amsmath}
\usepackage{amssymb}
\usepackage{amsfonts}
\usepackage{amsthm}
\usepackage{amscd}
\usepackage{mathdots}
\usepackage{mathtools}
\usepackage{tikz}
\usepackage{tikz-cd}
\usepackage{pgfplots}
\usepackage{dynkin-diagrams}
\usetikzlibrary{positioning}
\pgfplotsset{compat=1.8}
\usepackage{subfig}
\usepackage[shortlabels]{enumitem}
\usepackage{bbold}
\usepackage{fancyhdr}
\usepackage[title]{appendix}
\usepackage[noblocks]{authblk}
\usepackage{verbatim}
\makeatletter
\def\revddots{\mathinner{\mkern1mu\raise\p@
\vbox{\kern7\p@\hbox{.}}\mkern2mu
\raise4\p@\hbox{.}\mkern1mu\raise7\p@\hbox{.}\mkern1mu}}
\makeatother

\numberwithin{equation}{section}

\newcommand*{\defeq}{\mathrel{\vcenter{\baselineskip0.5ex \lineskiplimit0pt
                     \hbox{\scriptsize.}\hbox{\scriptsize.}}}%
                     =}
                     
\theoremstyle{plain}
\newtheorem{thm}{Theorem}[section]
\newtheorem{lem}[thm]{Lemma}
\newtheorem{prop}[thm]{Proposition}
\newtheorem{cor}[thm]{Corollary}

\theoremstyle{definition}
\newtheorem{defn}[thm]{Definition}

\newtheorem{rem}[thm]{Remark}

\theoremstyle{remark}

\newcommand{\F}{F}
\newcommand{\GL}{\textnormal{GL}}

\newcommand{\GO}{\textnormal{GO}}
\newcommand{\GU}{\textnormal{GU}}

\newcommand{\SO}{\textnormal{SO}}
\newcommand{\PU}{\textnormal{PU}}
\newcommand{\U}{\textnormal{U}}

\newcommand{\Ind}{\mathrm{Ind}}
\newcommand{\Hom}{\mathrm{Hom}}
\newcommand{\Ext}{\mathrm{Ext}}

\newcommand{\Tr}{\mathrm{Tr}}
\newcommand{\tr}{\mathrm{tr}}
\newcommand{\SL}{\mathrm{SL}}

\newcommand{\Irr}{{\textnormal{Irr}}}
\newcommand{\cind}{{\textnormal{c-ind}}}
\newcommand{\Spin}{{\textnormal{Spin}}}

\title{Howe duality for a quasi-split exceptional dual pair}
\author{Petar Bakić, Gordan Savin}
\date{}
\begin{document}

\maketitle

\begin{abstract}
We prove Howe duality for the theta correspondence arising from the $p$-adic dual pair $G_2 \times (\PU_3 \rtimes \mathbb{Z}/2\mathbb{Z})$ inside the adjoint quasi-split group of type $E_6$.
\end{abstract}

\section*{Introduction}

Let $F$ be a $p$-adic field, that is, a non-archimedean local field of characteristic 0. Simple exceptional Lie algebras over $F$ can be constructed from
pairs $(\mathbb O, J)$ where $\mathbb O$ is an octonion algebra over $F$, and $J$ a Freudenthal Jordan algebra. Let 
$G=\mathrm{Aut}(\mathbb O)$ and $G'=\mathrm{Aut}(J)$.  
Let $\mathfrak g$ and $\mathfrak g'$ be the Lie algebras of $G$ and $G'$, respectively.  Then, 
by a construction of Tits \cite{Ja}, 
\[ 
\mathfrak h = \mathfrak g\oplus \mathfrak g' \oplus \mathbb O^{\circ} \otimes J^{\circ} 
\]
has a structure of a simple exceptional Lie algebra over $F$, where $ \mathbb O^{\circ} $ and $J^{\circ} $ denote trace 0 elements in 
$\mathbb O $ and $J$, respectively. Let $H=\mathrm{Aut}(\mathfrak h)$. It is evident from the construction that there is an inclusion 
\[ 
G \times G' \subseteq H . 
\] 
The group $G$ is a split exceptional group of type $G_2$, whereas $G'$ and $H$ depend on $J$.  
A Freudenthal Jordan algebra is a form of $J_3(C)$, the algebra of $3\times 3$ Hermitian symmetric matrices 
with coefficients in a composition $F$-algebra $C$, see Chapter IX in \cite{KMRT}.  A composition algebra, 
roughly speaking, is a non-associative algebra with an anti-involution $x\mapsto \bar x$ such that 
$N_C(x)=x\bar x$ is a quadratic form satisfying composition, that is, $N_C(xy)=N_C(x) N_C(y)$ for all $x,y\in C$. The case treated in this paper is $C=K$, a quadratic field 
extension of $F$. Then 
\[ 
G'=\mathrm{PU}_3(K) \rtimes\mathrm{Gal}(K/F),
\] 
 where $\mathrm{PU}_3(K)$ is the quotient of the unitary group $\mathrm U_3(K)$ in 
three variables by its center, and $\mathrm{Gal}(K/F)$ acts on coefficients of $\mathrm U_3(K)$ naturally.  The group $H$ is quasi-split of absolute type $E_6$.   

\smallskip  

Let $\Pi$ be the minimal representation of $H$. The goal of this paper is to understand the restriction of $\Pi$ to the dual pair $G\times G'$. More precisely, let 
$\pi$ be a smooth, irreducible representation of $G$. Then there exists a smooth representation $\Theta(\pi)$ of $G'$ such that $\pi\boxtimes \Theta(\pi)$ is the maximal $\pi$-isotypic quotient of $\Pi$. If $\Theta(\pi)$ is non-zero, we prove that it is a finite length $G'$-module, and that it has a unique irreducible quotient 
$\theta(\pi)$. Conversely, if $\sigma$ is an irreducible representation of $G'$ then  $\Theta(\sigma)$ is a finite length $G$-module and, it it is non-zero, then it 
has a unique irreducible module $\theta(\sigma)$. The results are summarized in Theorem \ref{thm_Howe}.

\smallskip 
These results are proved by a period ping-pong, introduced in \cite{gan2021howe}, that can be viewed as a generalization of the doubling method for classical 
theta correspondences \cite{Howe_Corvallis}, \cite{kudla1996notes}. Here, just as for classical theta correspondences, one needs the following ingredient:  
 If $\sigma$ is an irreducible quotient of $\Theta(\pi)$, then $\sigma^{\vee}$ is a quotient of $\Theta(\pi^{\vee})$, where $\pi^{\vee}$ denotes the smooth dual of $\pi$.  For classical theta correspondences this statement can be obtained using the M\oe glin--Vigneras--Waldspurger involution \cite{MVW}. Existence of such an involution is a non-trivial matter; however, if $\pi$ is tempered then $\pi^{\vee}$ is isomorphic to the complex conjugate $\bar \pi$. Since $\bar \Pi\cong \Pi$, it follows at once that $\Theta(\bar\pi)$ is the complex conjugate of $\Theta(\pi)$. Thus the method of period ping-pong works well for tempered representations; however, separate treatment is needed for non-tempered representations. These representations are realized as Langlands quotients of principal series representations and here the method of Jacquet functors works well. Thus a principal contribution of this paper is a computation of the Jacquet functors of $\Pi$ with respect to maximal parabolic subgroups of $G$ and $G'$.  
 
 \smallskip 
 For non-tempered representations we obtain the following explicit result. 
 The group $G'$ is quasi-split of rank one. The Levi factor of a Borel subgroup is isomorphic to $K^{\times}\rtimes \mathrm{Gal}(K/F)$. 
 Let $\chi$ be a character of $K^{\times}$. Let $i(\chi)$ be 
 the two-dimensional representation of $K^{\times}\rtimes \mathrm{Gal}(K/F)$ obtained by inducing $\chi$. 
 Assume, for simplicity, that  $\chi$ is not $\mathrm{Gal}(K/F)$-invariant. Then $i(\chi)$ is irreducible and $i(\chi)\cong i(\chi')$ if and only if $\chi'$ is in the 
$\mathrm{Gal}(K/F)$-orbit of $\chi$.  Now, $i(\chi)$ defines a principal series representation $\sigma$ of $G'$. We now describe its theta lift to $\sigma$ to $G$. 
 The group $G$ has two conjugacy classes of maximal parabolic subgroups; we shall use the letters $Q_1$ and $Q_2$ for parabolic subgroups in the two classes, 
 where the unipotent radical of $Q_2$ is a two step nilpotent group, and the unipotent radical of $Q_1$ is a three step nilpotent group. The Levi factors of both parabolic 
 groups are isomorphic to $\mathrm{GL}_2(F)$.  
Let $W_F\subset W_K$ denote the Weil groups of $F$ and $K$. 
Recall, by local class field theory, that $W_K^{\mathrm{ab}} \cong K^{\times}$. Thus $\chi$ can be viewed as a character of $W_K$. We induce 
$\chi$ to $W_F$ and obtain a parameter of a supercuspidal representation $\tau$ of $\mathrm{GL}_2(F)$.  The theta lift of $\sigma$ is a representation of 
$G$ obtained by inducing $\tau$ from the maximal parabolic $Q_1$. 

\bigskip

The authors would like to thank Wee Teck Gan for initiating this project and for his continued interest. One of his letters to the authors has been adapted to form Section \ref{subs_mini} in the present paper. G.~Savin is partially supported by a grant from the National Science Foundation, DMS-1901745.

\section{Preliminaries}

\subsection{Basic number theory, {\footnotesize stealing title from Weil}}  
Let $F$ be a non-Archimedean local field with the absolute value $|\cdot|$ normalized as usual, and let $K/F$ be a quadratic extension. We let $\overline{z}$ denote the Galois conjugate of an element $z\in K$; we set $N_{K/F}(z) = z\cdot \overline{z}$ and $\Tr(z) = z + \overline{z}$. We let $\omega_{K/F}$ denote the character of $F^{\times}$ that corresponds to $K$ by local class field theory.  

Let $W_K \subset W_F$ be the Weil groups of $K$ and $F$ respectively. The quotient of $W_F$ by the commutator subgroup of $W_K$ is the relative Weil 
group $W_{K/F}$ 
\[ 
1\rightarrow K^{\times}  \rightarrow W_{K/F} \rightarrow \mathrm{Gal}(K/F)  \rightarrow 1. 
\] 
Let $D$ be the unique quaternion algebra over $F$. By Appendix III in \cite{We}  $W_{K/F}$ can be realized as the centralizer of $K^{\times}$ in 
$D^{\times}$. Thus 
\[ 
W_{K/F}= K^{\times} \cup K^{\times} j 
\] 
where $jz=\bar z j$ for all $z\in K^{\times}$  and $j^2$ is in $F^{\times}$, but not in the index two subgroup $N_{K/F}(K^{\times})$.   
Now any character $\chi$ of $K^{\times}$ defines a two dimensional representation of $W_{F}$
\[ 
\rho(\chi) = 
\mathrm{Ind}_{K^{\times}}^{W_{F}}(\chi).  
\] 

\begin{lem} 
\label{lem_base_change}Let $\sigma$ be a non-trivial element in $\mathrm{Gal}(K/F)$. 
\begin{enumerate} 
\item $\rho(\chi) \cong \rho(\chi')$ if and only if $\chi'=\chi$ or $\chi^{\sigma}$.  
\item $\rho(\chi)$ is irreducible if and only if $\chi\neq \chi^{\sigma}$. 
\item If $\chi= \chi^{\sigma}$ then $\rho(\chi)= \chi_1 \oplus \chi_2$ where $\chi_i$ are two characters of 
$F^{\times}$ such that $\chi_i(N_{K/F}(z))\allowbreak = \chi(z)$ for all $z\in K^{\times}$.  
\end{enumerate} 
\end{lem} 
\begin{proof} 
This is all a simple consequence of the explicit description of $W_{K/F}$.  For the last, observe that the condition 
$\chi= \chi^{\sigma}$, by Hilbert 90, implies that $\chi$ is trivial on norm one elements in $K^{\times}$, thus 
the formula $\chi_i(N_{K/F}(z)) = \chi(z)$ defines $\chi_1$ and $\chi_2$ on the index two subgroup of $F^{\times}$.  The two 
characters differ by the character $\omega_{K/F}$.
\end{proof}   

The determinant of $\rho(\chi)$ is an Asai character of $W_F$ denoted by $\mathrm{As}^-(\chi)$.  A character $\chi$ is called conjugate dual if 
$\chi^{-1}= \chi^{\sigma}$. Note that this implies that $\chi$ is trivial on $N_{K/F}(K^{\times})$. Thus the restriction of $\chi$ to $F^{\times}$ is either 
trivial or $\omega_{K/F}$.  Respectively, we say that $\chi$ is conjugate-orthogonal or conjugate-symplectic.  
The following lemma is now again a simple exercise, using the explicit description of $W_{K/F}$.  

\begin{lem} Assume that $\chi$ is a conjugate dual character of $K^{\times}$. Then $\mathrm{As}^-(\chi)=1$ is $\chi$ is conjugate-symplectic 
$\mathrm{As}^-(\chi)=\omega_{K/F}$ if $\chi$ is conjugate-orthogonal.  
\end{lem}

\subsection{Representations of $p$-adic groups}

Let $G$ denote the group of $F$-points of a reductive algebraic group \textbf{G}. We denote the category of smooth (complex) $G$-representations by $\mathcal{R}(G)$; the set of (equivalence classes of) irreducible representations of $G$ will be denoted by $\text{Irr}(G)$.

We recall the various functors which play a role in the representation theory of $p$-adic groups. Let $P$ be a parabolic subgroup of $G$ with Levi decomposition $P = MN$. We then have the parabolic induction functor $\Ind_P^G$, as well its normalized version, $i_P^G$. If $\pi$ is a smooth representation of $G$, we may consider the Jacquet functor $\pi \mapsto \pi_N$, where $\pi_N$ denotes the space of $N$-coinvariants of $\pi$. The normalized version of the Jacquet functor will be denoted by $r_P(\pi)$. Recall that parabolic induction is adjoint to the Jacquet functor. First, we have the (standard) Frobenius reciprocity, which states that there is a natural isomorphism
\[
\Hom_G(\pi, i_P^G(\sigma)) \cong \Hom_M(r_P(\pi),\sigma);
\]
here $\pi$ and $\sigma$ are representations of $G$ and $M$, respectively. Equally useful is the second (Bernstein) form of Frobenius reciprocity:
\[
\Hom_G(i_P^G(\sigma),\pi) \cong \Hom_M(\sigma,r_{\overline{P}}(\pi));
\]
here $\overline{P}$ denotes the parabolic subgroup opposite to $P$. Finally, we will occasionally use the compact induction functor, which we denote by $\cind$.

\subsection{Cubic Jordan algebras}
\label{subs_cubicJordan}

The space $J$ of Hermitian symmetric $3\times 3$ matrices over $K$ is a Jordan algebra with multiplication 
\[ 
x\circ y = \frac{1}{2} (xy+yx)= \frac{1}{2}[(x+y)^2 - x^2 - y^2]  
\] 
and identity $1$. A typical element of $J$ is
\[
\begin{pmatrix}
a & x & \overline{y}\\
\overline{x} & b & z\\
y & \overline{z} & c
\end{pmatrix},
\]
where $x,y,z \in K$ and $a,b,c\in F$. We let $J_{ij}$ (for $1 \leq i \leq j \leq 3$) denote the subspace of $J$ consisting matrices whose entries are $0$ except on the positions $(i,j)$ and $(j,i)$. For more details on the subject of cubic Jordan algebras, the reader can consult Chapter 38 in \cite{KMRT} and Chapter 4 in \cite{taste}. 
 
 \vskip 10pt 
 
 Let $N$ and $T$ denote the norm (determinant) and the (usual) trace of $3\times 3$ matrices. Recall that $N(x) =x x^{\#}$  where 
 $x^{\#}$ is the usual adjoint matrix to $x$, i.e. made of $2\times 2$ minors of $x$.  
  Let $(x,y,z)$ be the symmetric trilinear form on $J$ such that 
$(x,x,x) =6 N(x)$, that is, 
\[ 
(x,y,z)= N(x+y+z) - N(y+z) - N(x+z) - N(x+y) + N(x) + N(y) + N(z).  
\] 
Then $T(x)=\frac{1}{2}(x,1,1)$  and the adjoint $x^{\#}$ can be defined as the unique element in $J$ such that 
\[ 
(x,x,y) =(x^{\#}, y, 1) 
\] 
for all $y\in J$. A basic fact of linear algebra is that any $x\in J$ satisfies the characteristic 
 polynomial 
 \[
 x^3 - T(x) x^2 + T(x^{\#})x- N(x)=0. 
 \] 
 Multiplying this equation by $x^{\#}$ and then factoring out $N(x)$ gives 
 \[ 
 x^2 -T(x) x + T(x^{\#})- x^{\#}=0 
 \] 
 for all $x\in J$. 
This implies that $x^2$ and thus the Jordan multiplication is completely determined by the cubic form $N$ and the identity $1$. It follows at once that the 
 group of automorphisms of the Jordan algebra is equal to the group of automorphisms of the cubic pointed space $(N,1)$.


\vskip 10pt 
\noindent More generally, if $e\in J$ such that $N(e)=ee^{\#}=1$, we can define a Jordan multiplication 
\[ 
x\circ y = \frac{1}{2} (xe^{\#}y+ye^{\#}x). 
\] 
This is a Jordan algebra $J_e$ with identity $e$, and trace 
\[ 
T_e(x):=\frac{1}{2}(x,e,e) =  \frac{1}{2}(e^{\#}, x, 1)
\]

The trace pairing is given by
\[
(x,y) \mapsto T_e(x\circ y) = (e,e,x)(e,e,y)/4-(e,x,y) = T_e(x)T_e(y) - (e,x,y).
\]
For every $x\in J$ we can define $x^{\#_e}$ by 
\[ 
(x,x,y) =(x^{\#_e}, y, e) 
\] 
for all $y\in J$. Although we shall not need it, we record that $x^{\#_e}=e x^{\#}e$. Again, $J_e$ is determined by $N$ and $e$, thus the automorphism group of $J_e$ is
the group of automorphisms of the cubic pointed space $(N,e)$.  

 \vskip 10pt  
 Elements of $J$ are Hermitian symmetric matrices, in particular, any $e$ such that $N(e)=1$ defines a symmetric Hermitian form on $K^3$ of discriminant one. 
 Since $F$ is $p$-adic, any two such Hermitian spaces are isomorphic. It follows that the Jordan algebras $J_e$ are isomorphic. In the rest of this article, it will be convenient to fix
 \[
e = \begin{pmatrix}
 & & -1\\
  & -1 & \\
  -1 & & 
\end{pmatrix}.
\]
To simplify notation, we will write $J$ instead of $J_e$.

\vskip 10pt
\noindent Finally, we let $L_J$ denote the group of linear transformations of $J$ which preserve $N$. Then
\[
L_J \cong (\{g \in \GL_3(K): \det(g) \in U(1)\}/U(1)) \rtimes \mathbb{Z}/2\mathbb{Z}.
\]
Here $U(1)$ is embedded into $\GL_3(K)$ diagonally. The non-trivial element of $\mathbb{Z}/2\mathbb{Z}$ acts by transposition, and the action of $\GL_3(K)$ on $J$ is given by
\[
(g,X) \mapsto gXg^*, \quad \text{for } g\in L_J \text{ and }X \in J.
\]
Here, $g^*$ denotes the conjugate-transpose of $g$.

\subsection{Groups}
\label{subs_groups}
\noindent Here we describe the various groups that appear in this paper, including the quasi-split group $H$ of type $E_{6,4}$ and the dual pair $G \times G'$ we wish to study.
%
%

Let $H$ denote the adjoint quasi-split group of type $E_{6,4}$ defined over $F$, with splitting field $K$; its Dynkin diagram is given by
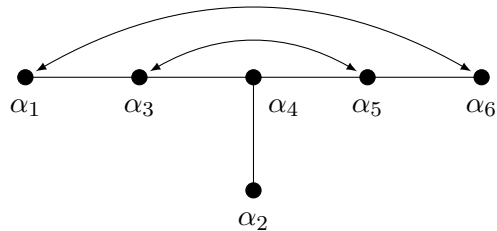
\[
\dynkin[%
ampersand replacement=\&, edge length=.75cm, labels*={\alpha_1,\alpha_2,\alpha_3,\alpha_4,\alpha_5,\alpha_6}, upside down = true, scale = 2, text style/.style = {scale = 1}, involutions={16;35}]E6
\]
The relative Dynkin diagram is
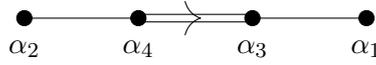
\[
\dynkin[%
ampersand replacement=\&, edge length=.75cm, labels*={\alpha_2,\alpha_4, \alpha_{3},\alpha_{1}}, upside down = true, scale = 2, text style/.style = {scale = 1}, arrow style={scale = 2.5}
]F4
\]
The groups we consider are best described on the level of Lie algebras. Here we follow the construction in \cite{rumelhart1997minimal}. Let $\mathfrak{h}$ denote the Lie algebra of $H$. By covering the vertex $\alpha_4$ in the above diagram, we see that $\mathfrak{h}$ contains a subalgebra $\mathfrak{h}_0 = \mathfrak{sl}_3 \oplus \mathfrak{l}$, where $\mathfrak{l} \cong \mathfrak{sl}_3(K)$ is the Lie algebra of $L_J$. Under the adjoint action of $\mathfrak{h}_0$, $\mathfrak{h}$ decomposes as
\[
\mathfrak{h} = \mathfrak{sl}_3 \, \oplus \, \mathfrak{l} \, \oplus \,W \otimes J \, \oplus \,  W^*\otimes J^*.
\]
Here $W = \langle e_1,e_2,e_3\rangle$ denotes the standard representation of $\mathfrak{sl}_3$, and $W^* = \langle e_1^*,e_2^*,e_3^*\rangle$ denotes its dual. We often use the trace pairing to identify $J^*$ with $J$. The Lie bracket relations are described in \cite{rumelhart1997minimal}.

Using the above description of $\mathfrak{h}$, we may now describe the dual pair $G_2 \times (\PU_3(K)\rtimes \mathbb{Z}/2\mathbb{Z})$. We let $\mathfrak{g}'$ denote the centralizer of $e$ in $\mathfrak{l}$: 
\[
\mathfrak{g}' = C_{\mathfrak{l}}(e).
\]
We then set
\[
\mathfrak{g} = C_{\mathfrak{h}}(\mathfrak{g}') = \mathfrak{sl}_3 \, \oplus\, W \otimes e \, \oplus\, W^* \otimes e^*.
\]
We let $G$ and $G'$ denote the closed subgroups of $H$ which correspond to $\mathfrak{g}$ and $\mathfrak{g}'$, respectively. Then $G\times G'$ is a dual pair inside $H$. Furthermore, $G$ is a split group of type $G_2$, and $G'$ is isomorphic to $\PU_3(K) \rtimes \mathbb{Z}/2\mathbb{Z}$. Indeed, we may describe $G'$ directly as the subgroup of $L_J$ which fixes $e$.

\vskip 10pt
\noindent The proof of Howe duality will require us to consider another dual pair inside $H$, which we now describe. Let $E$ be an \'{e}tale cubic $F$-algebra. We consider the set of $E$-isomorphism classes of embeddings $E \hookrightarrow J$. This set is in bijection with the set of ($E$-isomorphism) classes of twisted composition algebras $C$ such that $J = E \oplus C$; see Theorem 1.1 in \cite{gan2014twisted} for additional details. Fixing such a $C$, we let $i_C: E\to J$ denote an embedding in the corresponding isomorphism class; note that this also gives us embedding of $E^*$ into $J^*$. Let $G_{E,C}'$ denote the subgroup of $G'$ fixing $i_C$. The centralizer of $G_{E,C}'$ in $\mathfrak{h}$ contains
\[
\mathfrak{sl}_3 \, \oplus\, \mathfrak{t}_E \, \oplus\, W \otimes E \, \oplus\, W^* \otimes E^*,
\]
where $\mathfrak{t}_E$ is the Lie algebra of trace $0$ elements in $E$, and $E$ is embedded into $J$ using $i_C$. The above Lie algebra is isomorphic to $Lie(G_E)$, where $G_E$ is the simply-connected quasi-split group $\text{Spin}_8^E$. We thus get the dual pair $G_E \times G_{E,C}'$ inside $H$.

\subsection{Minimal representations and theta correspondence}
We will be interested in studying the minimal representation $\Pi$ of $H$. We recall one possible definition here. Let $\Pi$ be an irreducible representation of $H$. A result of Harish-Chandra then says that the character distribution of $\Pi$ can be expressed as
\[
\chi_\Pi = \sum_{\mathcal{O}} c_{\mathcal{O}}\hat{\mu}_{\mathcal{O}}
\]
where the sum is taken over all the nilpotent $H$-orbits in $\mathfrak{h}^*$, and $\hat{\mu}_{\mathcal{O}}$ is the Fourier transform of a (suitably normalized) $H$-invariant measure on $\mathcal{O}$. There exists a minimal non-trivial orbit $\textbf{O}_{\min}$ in $\mathfrak{h}(\overline{F})^*$. Assuming $\textbf{O}_{\min} \cap \mathfrak{h}^*$ consists of a single $H$-orbit $O_{\min}$, we have the following
\begin{defn}
We say that $\Pi$ is minimal if
\[
\chi_\Pi = c_0 + \hat{\mu}_{\mathcal{O}_{\min}}
\]
\end{defn}
\noindent For a detailed exposition of minimal representations and a construction of $\Pi$ for exceptional group, we refer the reader to \cite{gan2005minimal}.

Our goal is to study the restriction of $\Pi$ to the dual pair $G\times G'$ introduced above, and the exceptional theta correspondence which arises in this way. Fixing an irreducible representation $\pi$ of $G$, the maximal $\pi$-isotypic quotient of $\Pi$ is of the form
\[
\pi \otimes \Theta(\pi),
\]
for an admissible representation $\Theta(\pi)$ of $G'$ \cite[Lemme 2.III.4]{MVW}. This is the so-called big theta lift of $\pi$. Of course, one may start from $\sigma \in \Irr(G')$ to obtain the big theta lift $\Theta(\sigma)$ in the same way.

\section{Parabolic subgroups}

In this section we describe the three maximal parabolic subgroups of $H$ we consider in this paper.

\subsection{Three-step parabolic}
The first parabolic subgroup we consider is the maximal parabolic $P_1 = M_1N_1$ which corresponds to the vertex $\alpha_4$ of the Dynkin diagram. On the level of Lie algebras, it can be constructed as follows: let
\[
s = \begin{pmatrix}
1 & &\\
& 1 &\\
& & -2
\end{pmatrix} \in \mathfrak{sl}_3.
\]
Now set $\mathfrak{h}(i) = \{x \in \mathfrak{h}: [s,x] = ix\}$. Then the Lie algebra of $P_1$ is $\mathfrak{p}_1 = \mathfrak{m}_1 + \mathfrak{n}_1$, where
\[
\mathfrak{m}_1 = \mathfrak{h}(0) = \begin{pmatrix}
* & * & \\
* & * &\\
& & \phantom{*}
\end{pmatrix} \, \oplus \, \mathfrak{l}
\]
and $\mathfrak{n}_1 = \mathfrak{h}(1) +  \mathfrak{h}(2) + \mathfrak{h}(3)$ with
\begin{align*}
\mathfrak{h}(1) &= \langle e_1,e_2\rangle \otimes J,\\
\mathfrak{h}(2) &= \langle e_3^*\rangle \otimes J^*,\\
\mathfrak{h}(3) &= \begin{pmatrix}
\phantom{0} & \phantom{0} & *\\
&&*\\
&&
\end{pmatrix}  \subseteq \mathfrak{sl}_3.
\end{align*}
We also have
\begin{align*}
\mathfrak{h}(-1) &=\langle e_1^*, e_2^*\rangle \otimes J^*,\\
\mathfrak{h}(-2) &=\langle e_3\rangle \otimes J.
\end{align*}
Since $N_1$ is a $3$-step nilpotent group, we call $P_1$ the $3$-step parabolic.
We let $\Omega_i$ denote the minimal non-trivial $M_1$-orbit on $\mathfrak{h}(-i)$, $i=1,2$.

Setting $\mathfrak{u}_1(i) = \mathfrak{h}(i) \cap \mathfrak{g}$ for $i=1,2$, we see that
\begin{align*}
\mathfrak{u}_1(1) &= \langle e_1,e_2 \rangle \otimes  \langle e \rangle\\
\mathfrak{u}_1(2) &= \langle e_3^* \otimes e^* \rangle\\
\mathfrak{u}_1(3) &= \mathfrak{h}(3).
\end{align*}
Looking at the intersection of $P_1$ with $G\times G'$, we get
\[
(G\times G') \cap P_1 = Q_1 \times G'.
\]
Here $Q_1 = L_1U_1$ is the maximal parabolic subgroup of $G = G_2$; we identify the Levi factor $L_1$ with $\GL_2$ so that the action on $\mathfrak{u}_1(1)$ is the standard representation.

Now let $V_i$ be the orthogonal complement of $\mathfrak{u}_1(i)$ in $\mathfrak{h}(-i)$ (for $i=1,2$) with respect to the Killing form. Then
\begin{align*}
V_1 &=  \langle e_1^*,e_2^* \rangle \otimes  J_0^*\\
V_2 &=  \langle e_3 \rangle \otimes J_0.
\end{align*}
We need to describe $\Omega_i \cap V_i$, for $i=1,2$. We have (cf.~\cite{savin1999dual}, Lemma 4.1)
\begin{lem}
The group $L_1\times G'$ acts transitively on $\Omega_i \cap V_i$, $i=1,2$. Furthermore,
\begin{align*}
\Omega_1 \cap V_1 &= \{w^* \otimes X: w \in \langle e_1^*,e_2^* \rangle, X\in J_0^*, r(X)=1\}\\
\Omega_2 \cap V_2 &= \{e_3 \otimes X : X \in J_0, r(X)=1\}.
\end{align*}
\end{lem}

\subsection{Heisenberg parabolic}

Here we consider the maximal parabolic $P_2 = M_2N_2$ which corresponds to the vertex $\alpha_2$ of the Dynkin diagram. On the level of Lie algebras, it can be constructed as follows: let
\[
s = \begin{pmatrix}
1 & &\\
& 0 &\\
& & -1
\end{pmatrix} \in \mathfrak{sl}_3.
\]
Again, set $\mathfrak{h}(i) = \{x \in \mathfrak{h}: [s,x] = ix\}$. We intentionally abuse notation by reusing $s$ and $\mathfrak{h}(i)$, not only to reduce the number of unnecessary symbols, but also to emphasize the analogy in our constructions related to different parabolics. Since we never use these symbols for different parabolics at the same time, there is no fear of confusion.

The Lie algebra of $P_2$ is $\mathfrak{p}_2 = \mathfrak{m}_2 + \mathfrak{n}_2$, where
\[
\mathfrak{m}_2 = \mathfrak{h}(0) = \begin{pmatrix}
* & & \\
& * &\\
& & *
\end{pmatrix} \, \oplus \, \mathfrak{l} \, \oplus \, \langle e_2\rangle \otimes J  \, \oplus \, \langle e_2^*\rangle \otimes J^*,
\]
and $\mathfrak{n}_2 = \mathfrak{h}(1) +  \mathfrak{h}(2)$ with
\begin{align*}
\mathfrak{h}(1) &= \begin{pmatrix}
\phantom{0} & * & \phantom{0}\\
&&*\\
&&
\end{pmatrix} \, \oplus \, \langle e_1\rangle \otimes J  \, \oplus \, \langle e_3^*\rangle \otimes J^*,\\
\mathfrak{h}(2) &= \begin{pmatrix}
\phantom{0} & \phantom{0} & *\\
&&\\
&&
\end{pmatrix} \subseteq \mathfrak{sl}_3.
\end{align*}
We also note that
\[
\mathfrak{h}(-1) = \begin{pmatrix}
\phantom{0} & & \phantom{0}\\
*&&\\
&*&
\end{pmatrix} \, \oplus \, \langle e_3\rangle \otimes J  \, \oplus \, \langle e_1^*\rangle \otimes J^*.
\]
We often refer to $P_2$ as the Heisenberg parabolic, because its unipotent radical $N_2$ is a Heisenberg group with center $Z=\mathfrak{h}(2)$, attached to the symplectic space $N_2/Z = \mathfrak{h}(1)$. We let $\Omega$ denote the minimal non-trivial $M_2$-orbit on $\mathfrak{h}(-1)$. It is the orbit of a highest weight vector. 

We have
\[
\mathfrak{u}_2(1) = \mathfrak{g} \cap \mathfrak{h}(1) =  F  \, \oplus \, \langle e_1 \otimes e\rangle  \, \oplus \, \langle e_3^* \otimes e^*\rangle  \, \oplus \, F;
\]
we identify the intersection $\mathfrak{h}(1) \cap \mathfrak{sl}_3$ with $F \oplus F$.
Looking at the intersection of $P_2$ with $G\times G'$, we get
\[
(G\times G') \cap P_2 = Q_2 \times G'.
\]
Here $Q_2 = L_2U_2$ is the maximal parabolic subgroup of $G = G_2$ whose Levi factor $L_2$ we identify with $GL_2$ so that its action on $\mathfrak{u}_2(1)$ is the symmetric cube representation twisted by $\lvert\det\rvert^{-1}$ (see Section 3 in \cite{GGS}). Once again, $U_2$ is a Heisenberg group (with center $Z$) attached to the space $\mathfrak{u}_2(1)$.

Now let $V$ be the orthogonal complement of $\mathfrak{u}_2(1)$ in $\mathfrak{h}(-1)$ with respect to the Killing form. Then
\[
V = \langle e_3 \rangle \otimes J_0 \, \oplus \,  \langle e_1^* \rangle \otimes J_0^* \cong J_0 \, \oplus \, J_0,
\]
where $J_0$ denotes the set of all elements $X$ in $J$ such that $\tr(Xe) = 0$, and  $J_0^*$ is identified with $J_0$ using the trace pairing. Once again, we need to describe $\Omega \cap V$. Following the proof of Proposition 7.4 in \cite{magaard1997exceptional}, one shows the following
\begin{lem}
We have
\[
\Omega \cap V = \{(X,Y) \in J_0 \oplus J_0:\quad  r(X),r(Y) \leq 1,\quad \dim\langle X,Y\rangle  =1 \}.
\]
Furthermore, $L_2\times G'$ acts transitively on this set.
\end{lem}

\subsection{$B_3$ parabolic}
\label{subs_parabolicsD4}
Finally, we consider the maximal parabolic $P' = M'N'$ which corresponds to the vertex $\alpha_{1}$ of the relative Dynkin diagram. Set
\[
s = \begin{pmatrix}
1 & &\\
& 0 &\\
& & -1
\end{pmatrix} \in \mathfrak{sl}_3(K)
\]
and $\mathfrak{h}(i) = \{x \in \mathfrak{h}: [s,x] = ix\}$. Then the Lie algebra of $P'$ is $\mathfrak{p}' = \mathfrak{m}' + \mathfrak{n}'$. Here
\[
\mathfrak{m}' = \mathfrak{h}(0) = \mathfrak{sl}_3 \, \oplus \, \mathfrak{l}(0) \, \oplus \, W \otimes  \begin{pmatrix}
 &  & * \\
 & * &\\
*& & 
\end{pmatrix} \, \oplus \,  W^* \otimes  \begin{pmatrix}
 &  & * \\
 & * &\\
*& & 
\end{pmatrix}^*
\]
and $\mathfrak{n}_1 = \mathfrak{h}(1) +  \mathfrak{h}(2)$, with
\begin{align*}
\mathfrak{h}(1) &=\mathfrak{l}(1)\, \oplus \, W\otimes J_{12} \, \oplus \,  W^* \otimes J_{23}^* \\
\mathfrak{h}(2) &=\mathfrak{l}(2)\, \oplus \, W\otimes J_{11} \, \oplus \,  W^* \otimes J_{33}^*.\\
\end{align*}
We also have
\begin{align*}
\mathfrak{h}(-1) &=\mathfrak{l}(-1)\, \oplus \, W\otimes J_{23} \, \oplus \,  W^* \otimes J_{12}^*, \\
\mathfrak{h}(-2) &=\mathfrak{l}(-2)\, \oplus \, W\otimes J_{33} \, \oplus \,  W^* \otimes J_{11}^*,\\
\end{align*}
where
\[
\mathfrak{l}(-1) = \begin{pmatrix}
&&\phantom{*}\\
*&&\\
&*&
\end{pmatrix} \in \mathfrak{sl}_3(K) \quad \text{and} \quad
\mathfrak{l}(-2) = \begin{pmatrix}
&\phantom{*}&\phantom{*}\\
&&\\
*&&
\end{pmatrix} \in \mathfrak{sl}_3(K).
\]
We let $\Omega_i'$ denote the minimal non-trivial $M_1$-orbit on $\mathfrak{h}(-i)$, for $i=1,2$.

The intersection of $P'$ with $G\times G'$ is
\[
(G\times G') \cap P' = G \times B'.
\]
Here $B'$ denotes the semidirect product of $\mathbb{Z}/2\mathbb{Z}$ with the Borel subgroup consisting of all upper-triangular matrices in $\PU_3(K)$. There is a Levi decomposition $B' = T'U'$ with $T' = K^\times \rtimes \mathbb{Z}/2\mathbb{Z}$; we identify the diagonal torus in $\PU_3(K)$ with $K^\times$ using the isomorphism 
\[
\begin{pmatrix}
a & &\\
& b&\\
&&c
\end{pmatrix} \mapsto \dfrac{a}{b}
\]
($a,b,c$ satisfy $a\overline{c} = 1$ and $b\overline{b} = 1$).

Just like before, we let $\mathfrak{u}'(i) = \mathfrak{h}(i) \cap \mathfrak{g'}$ for $i=1,2$. We get
\begin{align*}
\mathfrak{u}'(1) &= \{\begin{pmatrix}
\phantom{0} & n &\\
&&-\overline{n}\\
&&
\end{pmatrix}\in \mathfrak{sl}_3(K)\} \cong K\\
\mathfrak{u}'(2) &= \{\begin{pmatrix}
\phantom{0} & \phantom{0} & y\\
&&\\
&&
\end{pmatrix}\in \mathfrak{sl}_3(K): \Tr(y) = 0\}.
\end{align*}
We let $V_i'$ be the orthogonal complement of $\mathfrak{u}'(i)$ in $\mathfrak{h}(-i)$ (for $i=1,2$) with respect to the Killing form. Direct computation shows that we have
\begin{align*}
V_1' &=  \{\begin{pmatrix}
&&\phantom{0}\\
x&&\\
&\overline{x}&
\end{pmatrix}: x \in K\}\, \oplus \, W \otimes J_{23} \, \oplus \,  W^* \otimes J_{12}^*,\\
V_2' &=  \{\begin{pmatrix}
&&\phantom{0}\\
\phantom{0}&&\\
y&&
\end{pmatrix}: y \in F\}\, \oplus \, W \otimes J_{33} \, \oplus \,  W^* \otimes J_{11}^*.
\end{align*}
It is convenient to use the following identifications (cf.~\cite[p.137]{savin1999dual}):
\[
V_1' = \mathbb{O}_0 \otimes K, \quad V_2' = \mathbb{O}_0\otimes  F;
\]
here we use $\mathbb{O}_0$ to denote the space of traceless octonions. Then $G\cong G_2$ acts naturally on $\mathbb{O}_0$, whereas $z\in K^\times \subset T'$ acts by $\dfrac 1z$ and $1/N_{K/F}(z)$ on $K$ and $F$, respectively. As before, we want to describe the $G\times T'$ orbits on $\Omega_i' \cap V_i'$ for $i=1,2$. However, a direct computation now shows that $\Omega_1' \cap V_1' = \emptyset$. On the other hand, we have (cf.~\cite{savin1999dual}, Lemma 2.4.11)
\begin{lem}
The group $G\times T'$ acts transitively on $\Omega_2' \cap V_2'$.
\end{lem}

We close this section with the following

\begin{rem}
\label{rem_Shalika}
One can work in a more general setting, starting with a Jordan algebra $J$  of Hermitian symmetric  $3\times 3$ matrices with 
coefficients in a composition algebra $C$. In particular, we have an exceptional Lie algebra 
\[ 
\mathfrak h = (\mathfrak{sl}_3(F) \oplus \mathfrak  l )\oplus W\otimes J \oplus (W\otimes J)^{\ast}.  
\] 
Observe that $\mathfrak{sl}_3(F) \subseteq \mathfrak l$,  where $x\in \mathfrak{sl}_3(F)$ acts on 
$y\in J$ by $xy+yx^{*}$ where $y^{*}$ is the transpose of $y$. This works even when $C$ is the non-associative algebra of octonions.   

In particular, one can define the parabolic subgroup $B'$ in $G'=\mathrm{Aut}(J)$ starting with the same choice of $s$ as above (note that $s \in \mathfrak{sl}_3(F) \subset \mathfrak{l}$). Let $U'$ be the unipotent radical of $B'$, and $U'(2)$ the center of $U'$. Then $U'(2)\cong C_0$, trace 0 elements in $C$, and 
$U'/U'(2)\cong C$. We will need this in \S \ref{sec_Howe} where we prove Howe duality.
\end{rem}

\section{Jacquet modules}
\label{sec_Jacquet}
Recall that $\Pi$ is the minimal representation of $H$. Our goal in this section is to compute $\Pi_{U_1}$, $\Pi_{U_2}$, and $\Pi_{U'}$.

\subsection{Three-step parabolic}
\label{subs_Jacquet1}
To compute $\Pi_{U_1}$, we begin by looking at $\Pi_{N_1}$. Recall that $N_1$ is a three-step nilpotent group: we have
\[
\{1\} = N_1(4) \subseteq N_1(3) \subseteq N_1(2) \subseteq N_1(1) = N_1
\]
with $N_1(i)/N_1(i+1) \cong \mathfrak{h}(i)$. This gives us a filtration of $\Pi$,
\begin{equation}
\label{eq_Pi_filtracija}
\{0\} = \Pi_4 \subseteq \Pi_3 \subseteq \Pi_2 \subseteq \Pi_1 \subseteq \Pi_0 = \Pi,
\end{equation}
where $\Pi_i/\Pi_{i+1} \cong (\Pi_i)_{N_1(i+1)}$. We have
\begin{gather*}
0 \to \Pi_3 \to \Pi  \to \Pi/\Pi_3 \to 0 \\
0 \to \Pi_2/\Pi_3 \to \Pi/\Pi_3 \to \Pi/\Pi_2 \to 0\\
0 \to \Pi_1/\Pi_2 \to \Pi/\Pi_2 \to \Pi/\Pi_1 \to 0.
\end{gather*}
We need to compute $\Pi/\Pi_1$, $\Pi_1/\Pi_2$, and $\Pi_2/\Pi_3$. The first quotient is simply $\Pi_{N_1}$; the remaining two can be computed using the work of M\oe glin and Waldspurger \cite{moeglin1987modeles}. We provide only a rough outline here; see e.g.~\cite{savin1999dual} or \cite{magaard1997exceptional} for additional details.

Recall that $\Omega_1$ is the minimal $M_1$-orbit in $\mathfrak{h}(-1)$. Let $f_1$ be an arbitrary element of $\Omega_1$, and denote by $M_{f_1}$ its stabilizer in $M_1$. Then $f_1$ defines a character $\psi_{f_1}$ on $N_1/N_1(2)$. Then \cite{moeglin1987modeles} shows that the space $\Pi_{N_1,\psi_{f_1}}$ (i.e.~the maximal quotient of $\Pi$ on which $N_1$ acts by $\psi_{f_1}$) is $1$-dimensional; $M_{f_1}$ acts on it by a character $\delta_1$. In short, we get
\[
\Pi_1/\Pi_2 = \cind_{M_{f_1}N_1}^{P_1}(\delta_1 \otimes \psi_{f_1})  = C_c(\Omega_1).
\]
We compute $\Pi_2/\Pi_3$ similarly. We choose an arbitrary element $f_2 \in \Omega_2$ and we let $M_{f_2}$ denote its stabilizer in $M_1$. Again, the results of \cite{moeglin1987modeles} (see also \cite[\S 5]{savin1999dual}) show that $f_2$ defines a certain Heisenberg representation, which we denote by $W_{f_2}$. We get
\[
\Pi_2/\Pi_3 = \cind_{M_{f_2}N_1}^{P_1}(W_{f_2}) = C_c(\Omega_2; W_{f_2}).
\]
Having computed the $N_1$-coinvariants, we proceed to investigate the $U_1$-coinvariants. The unipotent radical $U_1$ of $Q_1$ inherits the filtration from $N_1$:
\begin{equation}
\label{eq_U1_filtration}
\{0\} = U_1(4) \subseteq U_1(3) \subseteq U_1(2) \subseteq U_1(1) = U,
\end{equation}
where $U_1(i) = U_1 \cap N_1(i)$. In particular, $U_1(3) = N_1(3)$. We apply the $U_1$-coinvariants functor to the exact sequences above. From the first one, we see that $\Pi_{U_1} = (\Pi/\Pi_3)_{U_1}$. The remaining two sequences become
\begin{gather*}
0 \to (\Pi_2/\Pi_3)_{U_1} \to \Pi_{U_1} \to (\Pi/\Pi_2)_{U_1} \to 0\\
0 \to (\Pi_1/\Pi_2)_{U_1} \to (\Pi/\Pi_2)_{U_1} \to (\Pi/\Pi_1)_{U_1} \to 0.
\end{gather*}
Thus, to determine $\Pi_{U_1}$, we need to describe $(\Pi/\Pi_1)_{U_1}$, $(\Pi_1/\Pi_2)_{U_1}$, and $(\Pi_2/\Pi_3)_{U_1}$.

\bigskip

\noindent First, notice that $(\Pi/\Pi_1)_{U_1} = (\Pi_{N_1})_{U_1} = \Pi_{N_1}$. The $N_1$-coinvariants can be computed following \cite[\S 4]{Savin1994}, and the exponents of $\Pi$ have been determined in Proposition 8.4 of \cite{gan2005minimal}. As an $L_1 \times G'$-module,
\[
\Pi_{N_1} =|\textrm{det}|^{\frac{7}{2}}\mathbb{1}_{A_2} \,\oplus\, \omega_{K/F}|\textrm{det}|^{\frac{7}{2}}\cdot\Pi_{A_1}.
\]
Recall that the Levi factor $M_1$ consists of two parts (which correspond to the two parts of the $F_4$ diagram one obtains by removing the vertex $\alpha_4$): $A_2$ and $A_1$. Here $\mathbb{1}_{A_2}$ is the trivial representation of $A_2$ (in this case, $\SL_3(K)$), whereas $\Pi_{A_1}$ is a principal series representation of  $A_1$ (i.e.~$\SL_2(F)$). Furthermore, $|\mathrm{det}|$ denotes the standard determinant of $L_1 \cong \GL_2$.

\bigskip

\noindent Next, we consider $(\Pi_1/\Pi_2)_{U_1}$. Just like in \cite[Lemma 2.2]{magaard1997exceptional}, we obtain
\[
(\Pi_1/\Pi_2)_{U_1} = C_c(\Omega_1\cap V_1).
\]
Recall that $\Omega_1 \cap V_1 = \{w^* \otimes X: w \in \langle e_1^*,e_2^* \rangle, X\in J_0^*, r(X)=1\}$. The stabilizer of a line in $\Omega_1 \cap V_1$ (excluding $0$) is a product of Borel subgroups. For the sake of concreteness, we consider the line through $e_1^* \otimes xx^*$, where $x = (1,0,0)^*$ (once more, we identify $J_0^*$ with $J_0$). The stabilizer of $e_1^*$ in $L_1 = \GL_2$ is the subgroup $\overline{B} = T\overline{U}$ consisting of all lower-triangular matrices in $\GL_2$ (recall that we are considering the action of $\GL_2$ on $W^*$, the dual of the standard representation); here $T$ denotes the diagonal torus. The stabilizer of $xx^*$ is the Borel subgroup $B' = T'U'$ of $G'$ introduced in \S\ref{subs_parabolicsD4}. The group $T\times T'$ acts transitively on the above line, which we identify with $\F^\times$: The action of
\[
\left(\begin{pmatrix}
a & \\
 & b
\end{pmatrix} , z\right) \in {T} \times K^\times
\]
on $C_c(F^\times)$ is translation by $a^{-1}\cdot N_{K/F}(z)$ (and $\mathbb{Z}/2\mathbb{Z}$ acts trivially).

We deduce that 
\[
C_c(\Omega_1\cap V_1) = i_{\overline{B} \times B'}^{L_1 \times G'}(\delta_1 \otimes C_c(F^\times))
\]
(normalized induction) where $\delta_1$ is a character of the diagonal torus in $L_1\cong \GL_2$ which is yet to be determined.

\bigskip

\noindent Finally, we determine $(\Pi_2/\Pi_3)_{U_1}$. Recall that $\Omega_2 \cap V_2 = \{e_3 \otimes X : X \in J_0, r(X)=1\}$ is a single $L_1 \times G'$-orbit. We simplify the notation by identifying $\langle e_3 \rangle \otimes J$ with $J$, keeping in mind that $L_1 = \GL_2$ acts on $\Omega_2 \cap V_2$ by $\det^{-1}$. We start by observing that
\[
C_c(\Omega_2;W_{f_2})_{U_1(2)} = C_c(\Omega_2 \cap V_2; W_{f_2}).
\]
Notice that $(L_1\times G') \cap M_{f_2} = RU$, where
\[
R = \{(g,(z,\tau))\in L_1\times (K^\times \rtimes \mathbb{Z}/2\mathbb{Z}): \det(g) = N_{K/F}(z)\}
\]
and $U$ is the unipotent radical. Therefore $C_c(\Omega_2 \cap V_2; W_{f_2}) = \cind_{RU}^{L_1\times G'} W_{f_2}$ and thus
\[
(\Pi_2/\Pi_3)_{U_1} = \cind_{RU}^{L_1\times G'} ((W_{f_2})_{U_1}).
\]
It remains to determine $(W_{f_2})_{U_1}$; to do that, we need an explicit model for $W_{f_2}$. With this in mind, we choose $f_2 = xx^* \in J$, with $x = (1,0,0)^*$. Following \cite{moeglin1987modeles} (see also \cite{savin1999dual}), we consider the alternating form on $\mathfrak{h}(1) \cong \langle e_1,e_2\rangle \otimes J$ given by
\[
(v\otimes X,w\otimes Y) = (v,w)\cdot(f_2,X,Y). 
\]
Here $(v,w)$ is the standard symplectic form on $\langle e_1,e_2\rangle$, and $(f_2,X,Y)$ is the natural trilinear form on $J$. With our choice of $f_2$, the kernel of the bilinear form $(f_2,X,Y)$ is
\[
\Delta = \{
\begin{pmatrix}
a & x & \overline{y}\\
\overline{x} & & \\
y &&
\end{pmatrix} \in J
\}.
\]
We let $\Delta^\bot$ denote the orthogonal complement of $\Delta$ in $J$:
\[
\Delta^\bot = \{\begin{pmatrix}
0 & 0 & 0\\
0 & b & z\\
0 & \overline{z} & c
\end{pmatrix} \in J\}.
\]
The corresponding quadratic form is given by $2\cdot\begin{vmatrix}
a & z\\
\overline{z} & b
\end{vmatrix}$. We fix the maximal isotropic subspace consisting of elements of the form
\[
e_1 \otimes \begin{pmatrix}
0 & 0 & 0\\
0 & 0 & z\\
0 & \overline{z} & 0 
\end{pmatrix} + e_1 \otimes \begin{pmatrix}
0 & 0 & 0\\
0 & 0 & 0\\
0 & 0 & b_1
\end{pmatrix} + e_2 \otimes \begin{pmatrix}
0 & 0 & 0\\
0 & 0 & 0\\
0 & 0 & b_2
\end{pmatrix}.
\]
With this choice of polarization, the action of $U_1/U_1(2) = \langle e_1,e_2\rangle \otimes \langle e \rangle$ is given by
\[
\Pi(u)f(z,b_1,b_2) = \psi(u_1b_2 - u_2b_1) f(z,b_1,b_2),
\]
where $u = (u_1e_1 + u_2e_2) \otimes e$. We see that $U_1/U_1(2)$ acts trivially on functions supported on the subspace
\[
\{e_1 \otimes 
\begin{pmatrix}
0 & 0 & 0\\
0 & 0 & z\\
0 & \overline{z} & 0 
\end{pmatrix}: z\in K\}.
\]
It follows that
\[
(W_{f_2})_{U_1/U_1(2)} = W,
\]
where $W$ is the Heisenberg representation associated with the symplectic space $ \langle e_1,e_2\rangle \otimes K$; here we identify $z\in K$ with 
\[\begin{pmatrix}
0 & 0 & 0\\
0 & 0 & z\\
0 & \overline{z} & 0 
\end{pmatrix}.
\]
Unraveling the definitions, we see that $w \in K^\times \subset T'$ acts on $z\in K$ by
\[
(w,z) \mapsto \dfrac{z}{w},
\]
and $\mathbb{Z}/2\mathbb{Z}$ acts by Galois conjugation. Recall that we also have a $\GL_2$ action on the orbit $\Omega_2 \cap V_2$: an element $g \in L_1 = \GL_2$ acts by $\det(g)^{-1}$. In summary, we have
\[
(\Pi_2/\Pi_3)_{U_1} = \cind_{RU}^{L_1\times G'} W.
\]
We now recall the work of Roberts \cite{roberts1996theta}: there is a Weil representation $\omega$ of $R$ which induces to a representation $\tilde{\omega}$ of $\tilde{R} = L_1 \times T'$. The correspondence which arises from $\tilde{\omega}$ can be thought of as a similitude version of the usual $\SL_2 \times \SO(K)$ corrsepondence. As noted before, we have $M_{f_2} \cap \tilde{R} = R$, so we are precisely in the situation studied in \cite{roberts1996theta}.

An application of Schur's Lemma now shows that the action of $R$ on $W$ is a twist (by a character of $R$) of the action of $R$ on $\omega$. However, as one checks directly, every character of $R$ is a restriction of a character of $\tilde{R}$. We thus get
\[
(\Pi_2/\Pi_3)_{U_1} = i_{L_1\times B'}^{L_1\times G'}(\eta \otimes \tilde{\omega} ),
\]
where $\eta$ is a character yet to be determined. In fact, in \S \ref{subs_oddsnends}, we show that $\eta$ is the trivial character. Thus, we have
\begin{prop}
\label{prop_3step}
As a representation of $L_1 \times G'$, $r_{U_1}(\Pi)$ has a filtration with successive (top to bottom) subquotients
\begin{gather}
\delta_{Q_1}^{-1/2}\Pi_{N_1}  = |\mathrm{det}| \cdot \mathbb{1}_{A_2} \,\oplus\, \omega_{K/F}|\mathrm{det}|\cdot \Pi_{A_1}\label{eq_3step_top}\tag{T1}\\
i_{\overline{B} \times B'}^{L_1 \times G'}(\delta_1 \otimes C_c(F^\times))\label{eq_3step_middle}\tag{M1}\\
i_{L_1\times B'}^{L_1\times G'}(\tilde{\omega}).\label{eq_3step_bottom}\tag{B1}
\end{gather}
Recall that $\mathbb{1}_{A_2}$ is the trivial representation of $A_2$, whereas $\Pi_{A_1}$ is a principal series representation of  $A_1$. Furthermore, $|\mathrm{det}|$ is the standard determinant of $L_1 \cong \GL_2$ and $\delta_{Q_1} = |\mathrm{det}|^5$ is the modular character of $Q_1$. The center of $L_1 \cong \GL_2$ acts trivially on $\Pi_{A_2}$ and $\Pi_{A_1}$.
In \S \ref{subs_oddsnends}, we show that $\delta_1 = 1 \otimes \omega_{K/F}$.
\end{prop}

\subsubsection{The Fourier--Jacobi period}
We digress slightly to describe the Fourier--Jacobi period of the minimal representation $\Pi$. Although it is not required for the main results of the present paper (i.e.~for the proof of Howe duality), the Fourier--Jacobi period becomes useful in various similar settings. Since the computation is similar to the one we just performed to obtain $\Pi_{U_1}$, we take a moment to describe it here.

Recall the filtration \eqref{eq_U1_filtration}: $U_1$ is a three-step nilpotent group, and the quotient $U_1/U_1(3)$ is a three-dimensional Heisenberg group. Let $\psi$ be a character of its center $U_1(2)/U_1(3)$. Our goal is to describe the space of $\psi$-twisted coinvariants
\[
\Pi_{U_1(2),\psi},
\]
i.e.~the maximal quotient of $\Pi$ on which $U_1(2)/U_1(3)$ acts by $\psi$. Notice that this is (in addition to being a $G'$ module) a module for the Jacobi group, $FJ =Q_1^{\mathrm{der}}/U_1(3)$. Here $Q_1^{\mathrm{der}}$ denotes the derived group of $Q_1$.

Using the notation from \eqref{eq_Pi_filtracija}, we have $\Pi_{U_1(2),\psi} = (\Pi_2/\Pi_3)_{U_1(2),\psi}$. Recall that $\Pi_2/\Pi_3 = C_c(\Omega_2; W_{f_2})$. To find the $U_1(2)$ coinvariants we looked at $C_c(\Omega_2\cap V_2; W_{f_2})$, where $V_2$ was the orthogonal complement of $\mathfrak{u}_1(2)$ in $\mathfrak{h}(-2)$; this was identified with the space $J_0$ of traceless elements in $J$. However, we are now looking for the \emph{twisted} coinvariants. We get
\[
C_c(\Omega_2; W_{f_2})_{U_1(2),\psi} = C_c(\Omega_2\cap J_1; W_{f_2}).
\]
Here $J_1$ is the set of trace $1$ elements in $J$. One checks that $G'$ has two orbits on $\Omega_2\cap J_1$ (the set of rank $1$, trace $1$ elements in $J$). Thus, as a representation of $G' = \PU_3(K) \rtimes \mathbb{Z}/2\mathbb{Z}$, the space $C_c(\Omega_2\cap J_1; W_{f_2})$ splits into a direct sum of two induced representations. Indeed, one choice of representatives for these orbits is
\[
i = \begin{pmatrix}
\phantom{0} & & \phantom{0}\\
\phantom{0} & 1 & \phantom{0}\\
\phantom{0} & & \phantom{0}
\end{pmatrix}, \qquad j =
\begin{pmatrix}
\varepsilon & & 1/2\\
\phantom{0} & \phantom{ } & \phantom{0}\\
1/2 & & 1/(4\varepsilon)
\end{pmatrix}
\]
where $\varepsilon$ is an element of $F^\times$ which is not in the image of the norm map $N:K^\times \to F^\times$. (Indeed, $j$ cannot be written as $xx^*$ for some $x \in K^3$, which shows that it is not in the same orbit as $i$). The stabilizer of $i$ (resp.~$j$) in $\PU_3(K)$ is a unitary group in two variables, which we denote by $\text{U}(2)_i$ (resp.~$\text{U}(2)_j$). We thus have
\[
\Pi_{U_1(2),\psi} = i_{\text{U}(2)_i\rtimes \mathbb{Z}/2\mathbb{Z}}^{G'}(W_i)\ \oplus \ i_{\text{U}(2)_i\rtimes \mathbb{Z}/2\mathbb{Z}}^{G'}(W_j).
\]
Here $W_i$ (resp.~$W_j$) denotes the corresponding Weil representation, i.e.~the fiber at $i$ (resp.~$j$).

For example, the stabilizer of $i$ in $\PU_3(K)$ consists of all elements of the form
\[
\begin{pmatrix}
* & & *\\
 & 1 & \\
* & & *
\end{pmatrix} \in \text{U}_3(K).
\]
Thus we may identify the group $U(2)_i$ with the unitary group
\[
\{g \in \GL_2(K): \quad g\begin{pmatrix}
& 1\\
1 & 
\end{pmatrix}g^* = \begin{pmatrix}
& 1\\
1 & 
\end{pmatrix}\}
\]
in the obvious way. We obtain an explicit model of $W_i$ the same way we found $W_{f_2}$ above. Here it is convenient to fix $f_2 = i$; then
\[
\Delta = \{
\begin{pmatrix}
 & x & \\
\overline{x} & b & z\\
 & \overline{z} &
\end{pmatrix} \in J
\}\quad \text{and} \quad 
\Delta^\bot = \{\begin{pmatrix}
a & \phantom{0} & \overline{y}\\
 &  & \\
y &  & c
\end{pmatrix} \in J\}.
\]
Here $\Delta^\bot$ can be identified with the space $
I = \{
\begin{pmatrix}
a  & \overline{y}\\
y & c
\end{pmatrix}\}$ of $2\times 2$ Hermitian matrices. Thus the representation $W_i$ can be realized on the space $C_c(I)$, where the action of $U(2)_i$ on $A \in I$ is given by $(g, A) \mapsto gAg^*$.


\subsection{Heisenberg parabolic}
We compute $\Pi_{U_2}$ using the same general approach. Since $N_2$ is a Heisenberg group, there are only two subquotients we need to consider: $\Pi/\Pi_1$ and $\Pi_1/\Pi_2$. Just like in the case of the three-step parabolic, we have $(\Pi/\Pi_1)_{U_2} = (\Pi_{N_2})_{U_2} = \Pi_{N_2}$. Again, we may compute the $N_2$-coinvariants following \cite[\S 4]{Savin1994}. We get
\[
\Pi_{N_2} =|\mathrm{det}|^{\frac{3}{2}}\cdot \Pi_{C_{3}} \,\oplus\, \omega_{K/F}|\mathrm{det}|^2.
\]
Here $\Pi_{C_{3}}$ denotes the minimal representation of $M_2$ (corresponding to the $F_4$ diagram with the $\alpha_2$ vertex removed).

The description of $(\Pi_1/\Pi_2)_{U_2}$ is similar to the one we had in the three-step case: we have
\[
(\Pi_1/\Pi_2)_{U_2} = C_c(\Omega\cap V).
\]
Recall that $\Omega \cap V$ is a single orbit for $L_2 \times G'$, and that
\[
\Omega \cap V = \{(X,Y) \in J_0 \oplus J_0:\quad  r(X),r(Y) \leq 1,\quad \dim\langle X,Y\rangle  =1 \}.
\]
For the sake of concreteness, we consider the line through $e_3 \otimes xx^* \in \Omega \cap V$, where $x = (1,0,0)^*$. The stabilizer of this line is again the product $\overline{B} \times B'$ of Borel subgroups (here we are abusing notation by using $\overline{B} = T\overline{U}$ to denote the group of all lower-triangular matrices, but this time in $L_2$). The group $T\times T'$ acts transitively on the above line, which we identify with $\F^\times$: The action of
\[
\left(\begin{pmatrix}
a & \\
 & b
\end{pmatrix} , z\right) \in {T} \times K^\times
\]
on $C_c(F^\times)$ is translation by $a\cdot N_{K/F}(z)$. We deduce that 
\[
C_c(\Omega_1\cap V_1) = i_{\overline{B} \times B'}^{L_1 \times G'}(\delta_2 \otimes C_c(F^\times))
\]
where $\delta_2$ is a character yet to be determined. To summarize, we have
\begin{prop}
\label{prop_Heisenberg}
As a representation of $L_2 \times G'$, $r_{U_2}(\Pi)$ has a filtration with successive (top to bottom) subquotients
\begin{gather}
\delta_{Q_2}^{-1/2}\Pi_{N_2} = \Pi_{C_{3}} \,\oplus\, \omega_{K/F}\cdot|\mathrm{det}|^\frac{1}{2}.\label{eq_Heis_top}\tag{T2}\\
i_{\overline{B} \times B'}^{L_2 \times G'}(\delta_2 \otimes C_c(F^\times)).\label{eq_Heis_bottom}\tag{B2}
\end{gather}
Here $\delta_{Q_2} = |\mathrm{det}|^3$ denotes the modular character of $Q_2$. The center of $L_2 \cong \GL_2$ acts trivially on $\Pi_{C_{3}}$. In \S \ref{subs_oddsnends} we show that $\delta_2 = 1 \otimes \omega_{K/F}$
\end{prop}

\subsection{$B_3$ parabolic}
\label{subs_B3parabolic}
This case is entirely analogous to the previous two, so we just briefly sketch the results.

First, the top part in the filtration of $\Pi_{U'}$ is simply $\Pi_{N'}$.
Secondly, recall that in this case $\Omega_1'$ does not intersect $V_1'$, so the middle part of the filtration vanishes. The computation of the bottom (subrepresentation) part is equivalent to the one we described in detail in \S \ref{subs_Jacquet1}; in fact, the bottom part is induced from the same representation as the bottom part in Proposition \ref{prop_3step}. We omit the details and simply state the results:
\begin{prop}
\label{prop_D4}
As a representation of $G \times T'$, $r_{U'}(\Pi)$ has a filtration with successive (top to bottom) subquotients
\begin{gather}
\delta_{B'}^{-1/2}\Pi_{N'}  = \Pi_{B_3} \cdot |N_{K/F}(z)|\label{eq_D4_top}\tag{T3}\\
i_{L_1 \times T'}^{G \times T'}(\tilde{\omega}).\label{eq_D4_bottom}\tag{B3}
\end{gather}
Here $\delta_{B'} = |N_{K/F}(z)|^2$ denotes the modular character of $B'$. 
\end{prop}

We take a moment to describe the restriction of the representation $\Pi_{B_3}$ to $G \times T'$. Recall that $M'$ is the Levi factor of the parabolic $P'$ which corresponds to the $D_4$ part of the Dynkin diagram (i.e.~$B_3$ in the relative diagram). In other words, $M'$ is the group $D_4^E$, the simply connected quasi-split form of $\text{Spin}_8$ attached to the \'{e}tale cubic algebra $E=K+F$ ; these groups are described in detail in Section 2 of \cite{gan2014twisted}.

As mentioned above, $\Pi_{B_3}$ denotes the minimal representation of the group $M'$. Note that $K^\times \subset T'$ acts trivially on $\Pi_{B_3}$. However, the action of the Galois group $\mathbb{Z}/2\mathbb{Z}$ is non-trivial; in fact, this is precisely the situation studied in \cite{HMS} and \cite{GGJ}. There is a dual pair $G_2 \times S_E$ inside $D_4^E$ (here $S_E$ denotes the twisted form of $S_3$ attached to $E$). In \cite{GGJ}, the authors use the correspondence arising in this way to construct the so-called cubic unipotent A-packets of $G_2$. In our case, $S_E \cong \mathbb{Z}/2\mathbb{Z}$, so the corresponding A-packet contains two elements. One of them is a supercuspidal representation; the other is the Langlands quotient of $i_{Q_1}^{G_2}(\lvert\det\rvert\otimes\tau)$, with $\tau$ equal to the tempered representation $1 \times \omega_{K/F}$ of $\GL_2$. See Proposition 6.2 of \cite{GGJ} for a more detailed description of local A-packets arising in this way.

\subsection{Filling in the details}
\label{subs_oddsnends}
In this section, we determine the characters $\delta_1$, $\delta_2$, and $\eta$ that appear in the filtrations discussed above (recall that $\eta$ is introduced in the discussion preceding Proposition \ref{prop_3step}).

Both the bottom piece of $r_U'(\Pi)$ \eqref{eq_D4_bottom} and the bottom piece of $r_{U_1}(\Pi)$ \eqref{eq_3step_bottom} are induced from the Weil representation $\tilde{\omega}$. The similitude correspondence between $\GL_2(F)$ and $\GO(K) = K^\times \rtimes \mathbb{Z}/2\mathbb{Z}$ established by $\tilde{\omega}$ amounts to the usual base change $K^\times \longleftrightarrow \GL_2(F)$. Let $\chi$ be a character of $K^\times$.
By Lemma \ref{lem_base_change}, there are two possibilities:
\begin{enumerate}[(i)]
\item $\chi = \chi^\sigma$. In this case, $\chi$ extends to two characters of $K^\times \rtimes \mathbb{Z}/2\mathbb{Z}$, only one of which appears in the correspondence. We label that character $\chi^+$, and we let $\chi^-$ be the other one. Then $\chi^+$ lifts to the principal series $\chi_1 \times \chi_2$ (see Lemma \ref{lem_base_change}).
\item $\chi \neq \chi^\sigma$. In this case, $\chi$ lifts to a cuspidal representation $\rho_\chi$ of $\GL_2(F)$.
\end{enumerate}
See Section 7 of \cite{roberts1999nonarchimedean} for a brief account of this correspondence.

We now prove that the character $\eta$ introduced in the discussion preceding Proposition \ref{prop_3step} is trivial. First, note that $\eta$ is the restriction of a character $|\cdot|^s \otimes \xi$ of $\tilde{R}$, where $s\in \mathbb{R}$ and $\xi$ is a unitary character of $T'$. The above description of the $T' \leftrightarrow \GL_2(F)$ correspondence shows that
\[
|\cdot|^s\chi_1 \times |\cdot|^s\chi_2\quad \otimes\quad  i_{B'}^{G'}(\xi\chi^+)
\]
appears as a quotient of $i_{L_1\times B'}^{L_1\times G'}(\eta \otimes \tilde{\omega})$.
(To simplify notation, we write $\chi$ instead of $\chi^+$ in the rest of this section.) For a generic choice of $\chi$ we may assume that $
i_{Q_1}^G(|\cdot|^s\chi_1 \times |\cdot|^s\chi_2)$ and  $i_{B'}^{G'}(\xi\chi)$ are irreducible, so we get
\[
\Pi \twoheadrightarrow i_{Q_1}^G(|\cdot|^s\chi_1 \times |\cdot|^s\chi_2)\, \otimes\, i_{B'}^{G'}(\xi\chi).
\]
But now notice that $i_{B'}^{G'}(\xi\chi) = i_{B'}^{G'}(\xi^{-1}\chi^{-1}) =  i_{B'}^{G'}(\xi\cdot(\xi^{-2}\chi^{-1}))$ and we can apply the same reasoning to the character $\chi' = \xi^{-2}\chi^{-1}$.
This shows that
\[
|\cdot|^s\xi^{-2}\chi_1^{-1} \otimes |\cdot|^s\xi^{-2}\chi_2^{-1}
\]
needs to appear in the Jacquet module of $i_{Q_1}^G(|\cdot|^s\chi_1 \times |\cdot|^s\chi_2)$. Computing the said Jacquet module shows this to be possible only if $s=0$ and $\xi^2=1$.

It remains to prove that $\xi$ is trivial. Looking at the quotient \eqref{eq_Heis_top} in $r_{U_2}(\Pi)$ and applying Frobenius reciprocity, we see that
\[
\Pi \twoheadrightarrow i_{Q_2}^G(\omega_{K/F}\cdot |\textrm{det}|^\frac{1}{2})\, \otimes\, \mathbb{1},
\]
where $\mathbb{1}$ denotes the trivial representation of $\PU_3(K)$. The representation $i_{Q_2}^G(\omega_{K/F}\cdot |\textrm{det}|^\frac{1}{2})$ is of length $2$ (cf.~Proposition 4.1 in \cite{muic1997unitary}); from the above map we get
\[
\Pi  \twoheadrightarrow \pi \otimes \mathbb{1},
\]
where $\pi$ is the unique (Langlands) quotient of $i_{Q_1}^G(|\textrm{det}|\otimes (1\times \omega_{K/F}))$. Applying the Jacquet functor $r_{U_1}$ to the above map, 
and comparing wtih the subrepresentation \eqref{eq_3step_bottom} in the filtration $r_{U_1}(\Pi)$, we conclude that $\xi = 1$.

Once we have established that $\eta$ is the trivial character, it is not hard to determine $\delta_1$, the character that appears in the middle piece of the filtration \eqref{eq_3step_middle}. Recall that $\delta_1$ is a character of the diagonal torus in $L_1\cong \GL_2$. As explained above, the bottom piece of the filtration \eqref{eq_3step_bottom} shows that we have
\[
\label{eq_Pi_quotients}
\Pi \twoheadrightarrow i_{Q_1}^G(\chi_1 \times \chi_2) \otimes i_{B'}^{G'}(\chi). \tag{\textasteriskcentered}
\]
Here we are still assuming the choice of $\chi$ is such that both representations appearing on the right-hand side are irreducible. Applying the Jacquet module with respect to ${U_1}$, we see that, in addition to $\left(\chi_1 \times \chi_2\right) \otimes i_{B'}^{G'}(\chi)$ (which appears in the bottom piece of the filtration), $\Pi_{U_1}$ contains
\[
\left(\chi_1^{-1} \times \omega_{K/F}\right) \otimes i_{B'}^{G'}(\chi) \quad \text{and} \quad \left(\chi_2^{-1} \times \omega_{K/F}\right) \otimes i_{B'}^{G'}(\chi).
\]
These quotients come from the middle part of the filtration; in other words, we have
\[
i_{{\overline{B} \times B'}}^{{L_1} \times G'}(\delta_1 \otimes C_c(F^\times)) \twoheadrightarrow \left(\chi_i^{-1} \times \omega_{K/F}\right) \otimes i_{B'}^{G'}(\chi),
\]
for $i=1,2$. Using the Bernstein form of Frobenius reciprocity, and computing the appropriate Jacquet modules, we see that this is possible if and only if
\[
\delta_1 = 1 \otimes \omega_{K/F}.
\]
A similar argument can be used to determine $\delta_2$: we apply the Jacquet module $r_{U_2}$ to \eqref{eq_Pi_quotients}, observing that $i_{Q_1}^G(\chi_1 \times \chi_2) = i_{Q_2}^G(\chi_2 \times \omega_{K/F})$. Then certain quotients of $r_{U_2}(i_{Q_2}^G(\chi_2 \otimes \omega_{K/F}))$ come from the bottom of the filtration \eqref{eq_Heis_bottom}, and one verifies that $\delta_2 = 1\otimes \omega_{K/F}$.
\begin{rem}
\label{rem_middle_quotients}
The fact that $i_{\overline{B} \times B'}^{L_1 \times G'}(C_c(F^\times))$ is responsible for the \emph{two} quotients appearing above (even though we have a single orbit) is explained by the action of $\overline{B}\times B'$ on $C_c(F^\times)$. Recall that
\[
\left(\begin{pmatrix}
a & \\
 & b
\end{pmatrix} , (z,\tau)\right) \in {T} \times T'
\]
acts by $a^{-1}\cdot N_{K/F}(z)$, so we view this as an action of $F^\times \times N_{K/F}(K^\times)$. Since a character $\chi$ of $N_{K/F}(K^\times)$ (or equivalently, a Galois-invariant character of $K^\times$) extends to a character of $F^\times$ in two ways, both 
\[
\chi_1^{-1} \otimes \omega_{K/F} \quad \text{and} \quad \chi_2^{-1} \otimes \omega_{K/F}
\]
appear as quotients of $C_c(F^\times)$.
\end{rem}

\section{Howe duality}
\label{sec_Howe}

Our main result is the following theorem (Howe duality).
\begin{thm}
\label{thm_Howe}
\begin{enumerate}[(i)]
\item Let $\pi$ be an irreducible representation of $G_2$. If $\Theta(\pi) \neq 0$, then it is a representation of finite length, with a unique irreducible quotient $\theta(\pi)$.
\item For $\pi_1, \pi_2 \in \Irr(G_2)$,
\[
0 \neq \theta(\pi_1) \cong \theta(\pi_2) \Longrightarrow \pi_1 \cong \pi_2.
\]
\item If $\Theta(\pi) \neq 0$, then $\theta(\pi)$ is tempered if and only if $\pi$ is tempered. 
\item Let $\sigma \in \Irr(G')$. Then $\Theta(\sigma)$ is either $0$ or a representation of finite length.
\end{enumerate}
\end{thm}

The proof will take up the rest of this section; we provide an outline:

\begin{enumerate}[1)]
\item First, we consider the non-tempered correspondence in \S\ref{subs_nontemp}. In particular, we prove (i) and (ii) for non-tempered $\pi$; (iii) is then a consequence of the proof. These results will follow from our computations of Jacquet modules in \S \ref{sec_Jacquet}.
\end{enumerate}
Next, we study the lifts of tempered representations. If $\pi$ is tempered, we decompose $\Theta(\pi)$ into its cuspidal and non-cuspidal part: $\Theta(\pi) = \Theta(\pi)_{\text{c}} \oplus  \Theta(\pi)_{\text{nc}}$. We have the analogous decomposition $\Theta(\sigma) = \Theta(\sigma)_{\text{c}} \oplus  \Theta(\sigma)_{\text{nc}}$ for tempered $\sigma \in \Irr(G')$.
\begin{enumerate}[1), resume]
\item The finiteness of $ \Theta(\pi)_{\text{nc}}$ and $ \Theta(\sigma)_{\text{nc}}$ is proved in Proposition \ref{prop_nc_finite} using the Jacquet module computations from \S \ref{sec_Jacquet}.
\end{enumerate}
To analyze the cuspidal part we employ the strategy from \cite{gan2021howe}. The main idea is the “period ping-pong” introduced there\,---\,see Lemma \ref{lem_ppp1} and \ref{lem_pp2}.
\begin{enumerate}[1), resume]
\item We show that $\Theta(\pi)_{\text{c}}$ is either irreducible or zero in Proposition \ref{prop_no2tempered_gen} (for generic $\pi$) and Proposition \ref{prop_no2tempered_nongen} (for non-generic $\pi$). The uniqueness of the irreducible quotient in (i) is then deduced easily as a consequence of the period ping-pong; see Proposition \ref{prop_unique_quotient}. This proves (i).
\item Part (ii) is also shown to be a consequence of the period ping pong; see Prop.~\ref{prop_injectivity_tempered}.
\item Finally, the finiteness of $\Theta(\sigma)_{\text{c}}$ in (iv) follows from Propositions \ref{prop_no2_gen_quotients} and \ref{prop_nongeneric_quotient_implies_unique}.
\end{enumerate}

\subsection{Non-tempered correspondence}
\label{subs_nontemp}

Using the results of the Section \ref{sec_Jacquet}, we now compute the lifts of non-tempered representations. We begin by recalling the Langlands classification for $G=G_2$. Any non-tempered $\pi \in \Irr(G)$ is isomorphic to exactly one of the following representations:
\begin{enumerate}[a)]
\item Unique irreducible quotient of $i_{Q_1}^G(\tau)$ for $\tau = |\det|^s\tau_0$,  where $\tau_0$ is a tempered irreducible representation and $s>0$.
\item Unique irreducible quotient of $i_{Q_2}^G(\tau)$ for $\tau = |\det|^s\tau_0$,  where $\tau_0$ is a tempered irreducible representation and $s>0$.
\item Unique irreducible quotient of $i_{Q_1}^G(\tau)$, where $\tau$ is the unique (Langlands) quotient of $\chi_1 \times \chi_2$; here $|\chi_1| = |\cdot|^{s_1}$ and $|\chi_2| = |\cdot|^{s_2}$ with $s_1 > s_2 > 0$.
\end{enumerate}
In case of $\PU_3(K)$, the situation is even simpler. Any character of $K^\times$ can be written as a product $\chi\cdot |N_{K/F}|^s$ where $\chi$ is unitary and $s \in \mathbb{R}$. We let $I(\chi, s)$ denote the principal series representation of $\PU_3(K)$ obtained by inducing this character of $K^{\times}$. If $s>0$ this is a standard module and has a unique irreducible quotient. Before doing computations, we need to address the question of distinguishing extensions to $G'=\mathrm{PU}_3(K)\rtimes \mathrm{Gal}(K/F)$ of $\mathrm{Gal}(K/F)$-invariant representations of $\mathrm{PU}_3(K)$. Fortunately, for principal series the extension can be done at the level of inducing data. Given a $\mathrm{Gal}(K/F)$-invariant character $\chi$ of $K^{\times}$, only one extension to $K^{\times} \rtimes \mathrm{Gal}(K/F)$, denoted by $\chi^+$, appears in the quadratic base change (cf.\ref{subs_oddsnends}). Let $\chi^-$ denote the other extension. Thus, for Galois-invariant $\chi$, $I(\chi,s)$ extends to $G'$ in two ways: $I(\chi^+,s)$ and $I(\chi^-,s)$. When $\chi$ is not invariant, only one extension exists; in the following proposition, we denote it by $I(\chi^+,s)$ to enable uniform statements.

\begin{prop}\begin{enumerate}[(i)]
\label{prop_nontemp_lifts_1}
\item Let $\pi \hookrightarrow i_{Q_1}^G(\tau^\vee)$ with $\tau$ as in (a) above. 
If $\tau$ comes from a character $\chi\cdot |N_{K/F}|^s$ of $K^\times$ via base change $K^\times \to \GL_2(F)$ (with $s>0$ and $\chi$ unitary), then $\Theta(\pi)$ is a non-zero quotient of $I(\chi^+, s)$; in particular, it has finite length. If $\tau$ does not come from a character of $K^\times$ via base change, then $\pi$ does not appear in the theta correspondence.

Conversely, let $\sigma^+$ denote the unique irreducible quotient of $I(\chi^+, s)$. Then $\Theta(\sigma^+)\neq 0$.

\item Let $\pi \hookrightarrow i_{Q_2}^G(\tau^\vee)$ with $\tau$ as in (b) or (c) above. Then $\pi$ does not appear in the theta correspondence.
\end{enumerate}
\end{prop}

\begin{proof} We use the fact that $\Theta(\Pi)^* = \Hom_{G}(\Pi,\pi)$ (non-smooth linear dual). Thus, from $\pi \hookrightarrow i_{Q_i}^G(\tau^\vee)$ we get $\Theta(\Pi)^* = \Hom_{G}(\Pi,\pi) \subseteq \Hom_{G}(\Pi,i_{Q_i}^G(\tau^\vee)) = \Hom_{L_i}(r_{U_{i}}(\Pi), \tau^\vee)$ using Frobenius reciprocity. We now analyze the space $\Hom_{L_i}(r_{U_{i}}(\Pi), \tau^\vee)$ using Propositions \ref{prop_3step} and \ref{prop_Heisenberg}.
\begin{enumerate}[(i)]
\item Let $S_1, S_2$ and $S_3$ denote the subquotients appearing in \eqref{eq_3step_top}, \eqref{eq_3step_middle} and \eqref{eq_3step_bottom}, respectively. Comparing the central characters, one sees that $\Ext_{L_1}(S_1, \tau^\vee) = 0$ (recall that $s > 0$, so the central character is a negative power of $|\cdot|$). We thus get the following exact sequence:
\[
0 \to \Hom_{L_1}(S_2, \tau^\vee) \to \Hom_{L_1}(r_{U_1}(\Pi), \tau^\vee) \to \Hom_{L_1}(S_3, \tau^\vee)	\to \Ext_{L_1}(S_2, \tau^\vee).
\]
Recall that $S_2 = i_{\overline{B} \times B'}^{L_1 \times G'}(\delta_1 \otimes C_c(F^\times))$. Using the Bernstein form of Frobenius reciprocity we see that
\[
\Ext_{L_1}(S_2, \tau^\vee) = \Ext_{\GL_1 \times \GL_1}(\delta_1 \otimes C_c(F^\times), r_{B}(\tau^\vee)).
\]
Recall that the second $\GL_1$ factors acts on $\delta_1 \otimes C_c(F^\times)$ by $\omega_{K/F}$; by our assumption on $\tau$, this is different from the corresponding action on $r_{B}(\tau^\vee)$. Therefore $\Ext_{L_2}^i(S_2, \tau^\vee) = 0$ for all $i$, and the above long exact sequence becomes
\[
\Hom_{L_1}(r_{U_1}(\Pi), \tau^\vee) \cong \Hom_{L_1}(S_3, \tau^\vee).
\]
Now \eqref{eq_3step_bottom} shows that $S_3 = i_{L_1\times B'}^{L_1\times G'}(\tilde{\omega})$. By Lemma 9.4 of \cite{gan2011non}, the maximal $\tau^\vee$-isotypic quotient of $i_{B'}^{G'}(\tilde{\omega})$ is $\tau^\vee \otimes i_{B'}^{G'}(\Theta(\tau^\vee))$, where $\Theta(\tau^\vee)$ is the big theta lift of $\tau^\vee$ with respect to the similitude correspondence described in \S \ref{subs_oddsnends}. Hence, if $\Theta(\pi)$ is non-zero, then $\tau^\vee$ must come from a character of $K^\times$ via base change. Note that the character corresponding to $\chi_1 \times \chi_2$ (or $\rho_\chi$) is in fact $\chi^{-1}$ (and not $\chi$); this is accounted for by the fact that $w\in K^\times$ acts on $K$ by $1/w$ (see \S \ref{subs_Jacquet1}).
Thus, if $\tau = \chi_1|\cdot|^s \times \chi_2|\cdot|^s$ or $\tau = \rho_{\chi|\cdot|^s}$ for a unitary character $\chi$, we get $\Theta(\tau) = \chi^+|\cdot|^s$. 

This proves $\Theta(\pi)^* \subseteq \Hom_{L_1}(S_3,\tau^\vee) =I(\chi^+,s)^*$. Taking the smooth vectors (and the contragredient), we see that $\Theta(\pi)$ is a quotient of $I(\chi^+,s)$, as claimed. Furthermore, notice that the above proof shows that $I(\chi^+,s)$ is a quotient of $\Pi$, so $\Theta(\sigma)\neq 0$. 
\item This is proved by comparing the central character, the same way we did in (i). We omit the details.
\end{enumerate}
\end{proof}
Not surprisingly, we get analogous results for lifts from $G' \cong \PU_3(K)\rtimes \mathbb{Z}/2\mathbb{Z}$. The following proposition is proved just like Proposition \ref{prop_nontemp_lifts_1}, by analyzing $r_{U'}(\Pi)$:
\begin{prop}
\label{prop_nontemp_lifts_2}
As before, let $\sigma^+$ (resp.~$\sigma^-$) denote the unique irreducible quotient of $I(\chi^+, s)$ (resp.~$I(\chi^+, s)$), where $\chi$ is a unitary character of $K^\times$ and $s > 0$. Then $\Theta(\sigma^-)=0$, and $\Theta(\sigma^+)$ is a non-zero quotient of $i_{Q_1}^{G}(\tau)$, where $\tau$ is the representation of $\GL_2$ obtained from $\chi|\cdot|^s$ by base change $K^\times \to \GL_2(F)$. In particular, $\Theta(\pi)$ has finite length. Furthermore, $\Theta(\pi)\neq 0$, where $\pi$ is the unique irreducible quotient of $i_{Q_1}^G(\tau)$.
\end{prop}

\begin{rem}
\label{rem_temp2temp}
Notice that Propositions \ref{prop_nontemp_lifts_1} and \ref{prop_nontemp_lifts_2} combine to give us the following: Assume that $\pi$ and $\tau$ are irreducible representations of $G$ and $G'$, respectively, such that $\pi \boxtimes \tau$ is a quotient of $\Pi$. Then
\[
\pi \text{ is tempered} \iff \tau \text{ is tempered}.
\]
\end{rem}
\subsection{Finiteness of theta lifts}

Our first task is to prove that the big theta lift $\Theta(\pi)$ has finite length. To do this, we recall that $\Theta(\pi)$ can be decomposed as
\[
\Theta(\pi) = \Theta(\pi)_{\text{nc}} \oplus  \Theta(\pi)_{\text{c}},
\]
the sum of its non-cuspidal and cuspidal part. We first prove the following

\begin{prop}
\label{prop_nc_finite}
\begin{enumerate}[(i)]
\item Let $\tau \in \Irr(G')$ be tempered. Then $\Theta(\tau)_{\text{nc}}$ has finite length.
\item Let $\pi \in \Irr(G)$ be tempered. Then $\Theta(\pi)_{\text{nc}}$ has finite length.
\end{enumerate}
\end{prop}
\begin{proof}
(i) Recall that we have two maximal parabolic subgroups in $G$, $Q_i = L_iU_i$ for $i=1,2$. It suffices to show that the Jacquet module $r_{U_i}(\Theta(\tau))$ is a finite-length representation of $L_i$, i.e.~that the $\tau$-isotypic quotient of $r_{U_i}(\Pi)$ has finite length. To do that, we use the Jacquet module filtrations computed in Section \ref{sec_Jacquet}.

Consider $r_{U_1}(\Pi)$ first. Again, we let $S_1, S_2$ and $S_3$ denote the subquotients appearing in \eqref{eq_3step_top}, \eqref{eq_3step_middle} and \eqref{eq_3step_bottom}, respectively. We need to show that the multiplicity space of the $\tau$-isotypic quotient of $S_i$ has finite length, for $i=1,2,3$.

At the bottom, we have
\[
\Theta_{S_3}(\tau)^* \defeq \Hom_{G'}(S_3,\tau) \cong \Hom_{G'}(i_{L_1\times B'}^{L_1\times G'}(\tilde{\omega}), \tau) \cong \Hom_{T'}(\tilde{\omega}, r_{\overline{B}'}(\tau)).
\]
Now  $r_{\overline{B}'}$ is a finite-length representation of $T' \cong K^\times$. Taking any irreducible subquotient $\chi$ (which is in fact a character of $K^\times$), we have $\Hom_{T'}(\tilde{\omega}, \chi) \cong \Theta_{\tilde{\omega}}(\chi)^*$, where $\Theta_{\tilde{\omega}}(\sigma)$ is the theta lift of $\chi$ with respect to the Weil representation $\tilde{\omega}$. The fact that $\Theta_{\tilde{\omega}}(\chi)$ has finite length follows from the Howe duality theorem for classical (similitude) correspondences, cf.~\S \ref{subs_oddsnends}. This in turn shows (after taking the smooth vectors) that $\Theta_{S_3}(\tau)$ itself is of finite length as an $L_1$-module.

In the middle, we have
\begin{align*}
\Theta_{S_2}(\tau)^* \defeq \Hom_{G'}(S_2,\tau) &\cong \Hom_{G'}(i_{\overline{B}\times B'}^{L_1\times G'}(\delta_1 \otimes C_c(F^\times)), \tau)\\
&\cong \Hom_{T'}(i_{\overline{B}}^{G}(\delta_1 \otimes C_c(F^\times)), r_{\overline{B}'}(\tau)).
\end{align*}
Again, $r_B(\tau)$ is a finite-length representation of $T$. Taking an irreducible subquotient $\chi$ of $r_B(\tau)$, we see that
\[
\Hom_{T'}(i_{\overline{B}}^{G}(\delta_1 \otimes C_c(F^\times)), \chi) = i_{\overline{B}}^{G}(\chi_1^{-1} \otimes \omega_{K/F}) \oplus i_{\overline{B}}^{G}(\chi_2^{-1} \otimes \omega_{K/F})
\]
(see Remark \ref{rem_middle_quotients}); in particular, we get a representation of finite length. Taking the smooth vectors, we see that $\Theta_{S_2}(\tau)$ has finite length.

Finally, we need to check $S_1$, the top part of the filtration. However, $B'$ acts trivially on $S_1$, so the $\tau$-isotypic quotient is zero. 

The Jacquet module with respect to $U_2$ is analyzed in the same way. Let $S_1$ and $S_2$ be the subquotients of $r_{U_2}(\Pi)$ appearing in \eqref{eq_Heis_top} and \eqref{eq_Heis_bottom}, respectively. To show that the $\tau$-isotypic quotient of $S_2$ has finite length we proceed just like in the $U_1$ case; we omit the details. 

As for $S_1$, recall that $P_2 = M_2N_2$ is the Heisenberg parabolic in $H$. The Levi factor $M_2$ (which corresponds to the $C_3$ part of the relative diagram) has been described in \cite[\S 7.2]{gan2021twisted}\textemdash\,one can think of it roughly as a unitary group $\U_6(K)$\textemdash\,and $\Pi_{C_3}$ is its minimal representation. We thus need to analyze the $\tau$-isotypic quotient of $\Pi_{C_3}$ as a representation of $L_2 \cong \GL_2(F)$. On the other hand, we have the central isogeny $\GL_2(F) \times K^\times \to \GU_2(K)$. Thus the correspondence that arises from $\Pi_{C_3}$ can roughly be viewed as (the similitude version) of the classical correspondence
\[
\U_2(K) \longleftrightarrow \U_3(K)
\]
for unitary groups. It follows that the theta lift of $\tau$ with respect to $\Pi_{C_3}$ is a finite length representation of $\GL_2(F)$.

Part (ii) can be proved in the same way, by analyzing $r_{U'}(\Pi)$ as a $T'$-module. We let $S_1$ and $S_3$ denote the subquotients appearing in \eqref{eq_D4_top} and  \eqref{eq_D4_bottom}, respectively. To prove that the $\pi$-isotypic quotient of $S_3$ has finite length we repeat the arguments from the $U_1$ case; we leave the details to the reader. 

In order to analyze $S_3$ we need to consider the $\pi$-isotypic quotient of the representation $\Pi_{B_3}$. Recall that $\Pi_{B_3}$ is the minimal representation of the Levi factor $M'$. This Levi factor is a quasi-split group $D_4^E$ where $E = K\oplus F$, and we are thus looking at the correspondence for the dual pair $G_2 \times \text{Aut}(E)$ inside the group $D_4^E$. The finiteness now follows from the results of \cite{HMS}, where this correspondence has been studied in detail.

\end{proof}

Having established the finiteness of $\Theta(\pi)_{\text{nc}}$, we turn to $\Theta(\pi)_{\text{c}}$. Here our approach is based on the “ping-pong” of periods utilized in \cite{gan2021howe}. We will need to consider generic and non-generic representation separately. We begin by recalling the relevant periods.

\subsection{Shalika periods}

First, we recall the parabolic subgroup $B'$ of $G'$ discussed in Remark \ref{rem_Shalika}. The unipotent radical $U'$ of $B'$ has a filtration $\{0\} \subset U'(2) \subset U'$ with $U'(2)\cong C_0$ and $U'/U'(2)\cong C$. We let $\psi_{U'}$ be the character of $U'$ which, via the identification $U'/U'(2)\cong C$ is given by $\psi\circ Tr_{C/F}$. Finally, let $S$ be the semi-direct product of $U'$ and the stabilizer of of $\psi_{U'}$ in the Levi of $B'$.  
We note that the stabilizer is isomorphic to $\mathrm{Aut}(C)$. We denote by $\psi_S$ the (Shalika) character of $S$ equal to $\psi_{U'}$ on $U'$ and trivial on 
$\mathrm{Aut}(C)$.

Let $V$ be the unipotent radical of the Borel subgroup $Q_1\cap Q_2$ in $G$, and let $\psi_V: V \rightarrow \mathbb C^{\times}$ be a Whittaker character for $G=G_2$. Just like in \cite[Lemma 4.5]{SavinWeissman}, one shows that
\begin{equation}
\label{eq_Whittaker_period}
\Pi_{V,\psi_V} \cong \cind_S^{G'} \psi_S
\end{equation}
holds for general $C$. Here $\Pi_{V,\psi_V}$ denotes the maximal quotient of $\Pi$ on which $V$ acts by $\psi_V$. This immediately implies

\begin{cor}
\label{cor_genericTheta}
Let $\pi$ be an irreducible representation of $G'$ Then $\Theta(\pi)$ is (non-zero) generic if and only if $\pi$ has a non-trivial Shalika period. 
\end{cor}
For $C=K$ we have $\text{Aut}(C) = \mathbb{Z}/2\mathbb{Z}$, and the Shalika functional is simply the Whittaker functional extended trivially to $\mathbb{Z}/2\mathbb{Z}$. 

Conversely, recall that $Q_1$ is the three step maximal parabolic in $G$. Let $Q_1^{\mathrm{der}}$ be its derived group. In particular 
$Q_1^{\mathrm{der}}/U_1 \cong \SL_2(F)$, and $U_1/U_1(3)$ is a three-dimensional Heisenberg group.   Then 
\[
\Pi_{S,\psi_S} \cong \cind_{Q_1^{\mathrm{der}}}^{G} ( \rho_{\psi} \otimes \Theta(1)) 
\]
where $ \rho_{\psi}$ is the unique irreducible representation of $U_1/U_1(3)$, extended to $\widetilde \SL_2(F)$, and $\Theta(1)$ is the big theta lift of the trivial representation 
of $\mathrm{Aut}(C)$ to $\widetilde \SL_2(F)$, via the correspondence arising from $\widetilde \SL_2(F)\times \mathrm{Aut}(C)$ acting on the Weil representation on 
$C_c(C_0)$ given by $\psi$. 
If $C$ is the algebra of $2\times 2$ matrices, this is given by Proposition 11.6 in \cite{gan2021howe}. Of course, the proof generalizes. If $C=K$  then $\Theta(1)$ is an irreducible even Weil representation. 
The even Weil representation is a quotient of  a principal series representation $I_{\psi}(\chi)$ (notation of \cite[Section 11]{gan2021howe}) for a character $\chi$ such that $|\chi|=|\cdot|^{1/2}$.  Thus, 
by \cite[Corollary 11.4]{gan2021howe}, we have
\begin{equation}
\label{eq_Shalika}
\dim \,\Hom_G(\Pi_{S,\psi_S} , \pi) \leq 1 
\end{equation}
for any Whittaker generic and tempered irreducible representation $\pi$ of $G$.

\subsection{Howe duality for tempered representations}

We now have all the ingredients required for the ping-pong game:

\begin{lem}
\label{lem_ppp1}
Let $\Pi$ be the minimal representation of $H$. Let $\pi \in \Irr(G)$ be tempered, and let $\tau \in \Irr(G')$ be tempered such that
\[
\Hom_{G\times G'}(\Pi,\pi \boxtimes \tau) \neq 0.
\]
Then we have the following (natural) inclusions
\[
\Hom_V(\pi, \psi_V) \stackrel{(1)}{\subseteq} \Hom_V(\Theta(\tau), \psi_V)  \stackrel{(2)}{\cong} \Hom_S(\tau^\vee, \overline{\psi}_S)  \stackrel{(3)}{\subseteq} \Hom_S(\Theta(\pi^\vee), \overline{\psi}_S)  \stackrel{(4)}{\cong} \Hom_G(\Pi_{S,\overline{\psi}_S}, \pi^\vee).
\]
If $\pi$ is generic, than all the above spaces are one-dimensional. 
\end{lem}
\begin{proof}
This is analogous to Lemma 12.1 in \cite{gan2021howe}. First, (1) follows from $\Theta(\tau) \twoheadrightarrow \pi$. The isomorphism (2) follows from
\[
\Hom_V(\Theta(\tau), \psi_V)  \cong \Hom_{V\times G'}(\Pi, \psi_V\otimes \tau) \cong \Hom_{G'}(\Pi_{\psi_V}, \tau)
\]
combined with \eqref{eq_Whittaker_period}. Next, (3) follows from the fact that $\Theta(\overline{\pi})$ is the complex conjugate of $\Theta(\pi)$. Since $\overline{\pi} \cong \pi^\vee$ and $\overline{\tau} \cong \tau^\vee$, we have $\Theta(\pi^\vee) \twoheadrightarrow \tau^\vee$. Finally, (4) is
\[
\Hom_S(\Theta(\pi^\vee), \overline{\psi}_S) \cong \Hom_{G\times S}(\Pi,\pi^\vee \otimes \overline{\psi}_S) \cong \Hom_G(\Pi_{S,\overline{\psi}_S}, \pi^\vee).
\]
If $\pi$ is generic, then $\Hom_V(\pi,\psi_V)$ is one-dimensional. However, \eqref{eq_Shalika} shows that $\Hom_G(\Pi_{S,\overline{\psi}_S},\allowbreak \pi^\vee)$ is one-dimensional as well. The lemma follows.
\end{proof}
\noindent We now have two immediate consequences of the above lemma (cf.~Propoisition~12.2 and 12.3 of \cite{gan2021howe}):

\begin{prop}
\label{prop_no2_gen_quotients}
Let $\tau \in \Irr(G')$ be tempered. Then $\Theta(\tau)$ cannot have two irreducible tempered and generic quotients. 
\end{prop}
\begin{proof}
Assume that $\pi_1$ and $\pi_2$ are tempered and generic such that $\Theta(\tau) \twoheadrightarrow \pi_1 \oplus \pi_2$. Then $\dim \Theta(\tau)_	{V,\psi_V} \geq 2$. However, Lemma \ref{lem_ppp1} asserts that $\dim \Theta(\tau)_	{V,\psi_V} = 1$, so we have a contradiction.
\end{proof}

\begin{prop}
\label{prop_no2tempered_gen}
Let $\pi \in \Irr(G)$ be tempered and generic. Then $\Theta(\pi)$ cannot have two tempered irreducible quotients. In particular, its cuspidal part $\Theta(\pi)_c$ is either irreducible or zero.
\end{prop}

\begin{proof}
Let $\tau_1, \tau_2$ be irreducible and tempered such that $\Theta(\pi) \twoheadrightarrow \tau_1 \oplus \tau_2$. Lemma \ref{lem_ppp1} (applied to $\pi^\vee, \tau_1^\vee$, and again to $\pi^\vee, \tau_2^\vee$) implies
\[
1 = \dim \Hom_S(\tau_1,\psi_S) =  \dim \Hom_S(\Theta(\pi),\psi_S) =  \dim \Hom_S(\tau_2,\psi_S).
\]
But $\tau_1 \oplus \tau_2$ is a quotient of $\Theta(\pi)$, so we have
\[
1 =  \dim \Hom_S(\Theta(\pi),\psi_S)  \geq  \dim \Hom_S(\tau_1,\psi_S)  + \dim \Hom_S(\tau_2,\psi_S)  = 2,
\]
a contradiction.
\end{proof}

Thus, we have proved that $\Theta(\pi)_{\text{c}}$ has finite length in case $\pi$ is generic. To prove the same result for $\pi$ non-generic we need another version of period ping-pong, which we now describe. 

Recall the groups $G_E = \Spin_8^E$ and $G_{E,C}' = \text{Aut}(i_C:E\hookrightarrow J)$ introduced in \S\ref{subs_groups}. Together with $G$ and $G'$, they constitute a see-saw dual pair

\begin{center}
 \begin{tikzcd}[every arrow/.append style={dash}]
G_E = \Spin_8^E \arrow[d] \arrow[dr] & G' \arrow[d] \arrow[dl] \\
G_2 & G_{E,C}' = \text{Aut}(i_C:E\hookrightarrow J).
\end{tikzcd}
\end{center}

\noindent This gives us the standard see-saw identity
\[
\Hom_{G_{E,C}'} (\Theta(\pi), 1) = \Hom_{G_2}(R_C(E), \pi),
\]
where $R_C(E) = \Theta(1)$ denotes the big theta lift of the trivial representation of $G_{E,C}'$ to $G_E$.

To better understand the representations $R_C(E)$ (for various $C$), we need to relate them to a certain degenerate principal series of $G_E$. Here we use the results of \cite[\S 5]{gan2021howe}. Let $P_E = M_EN_E$ be the Heisenberg parabolic subgroup of $G_E = \Spin_8^E$. We consider the degenerate principal series
\[
I_{E,\omega_{K/F}}(s)  = \Ind_{P_E}^{G_E}(\omega_{K/F}|\det|^s).
\]
We then have
\begin{prop}
\label{prop_deg_ps_Spin}
Let $\pi \in \Irr(G_2)$ be tempered. Then
\begin{enumerate}[(i)]
\item $I_{E,\omega_{K/F}}(1/2) \twoheadrightarrow \bigoplus_C R_C(E)$, where the sum is taken over all $C$ such that $E \oplus C = J$.
\item $\Hom_{G_2}(I_{E,\omega_{K/F}}(1/2) , \pi ) = \Hom_{N_2}(\pi^\vee,\psi_E)$.
\end{enumerate}
\end{prop}
\begin{proof}
These results are taken, mutatis mutandis, from propositions 5.2 and 5.5 in \cite{gan2021howe}.
\end{proof}

The final ingredient we need is a description of the twisted $N_2$-coinvariants of $\Pi$:
\begin{lem}
\label{lem_GrS}
We have
\[
\Pi_{N_2,\psi_E} = \bigoplus_C \cind_{G_{E,C}'}^{G'}(1),
\]
where the sum is taken over all twisted composition algebras $C$ such that $E\oplus C = J$.
\end{lem}
\begin{proof}
This is essentially Lemma 2.9 in \cite{GrS}; the only difference is that here we have more than one isomorphism class of embeddings $E\hookrightarrow J$.
\end{proof}

We are now ready for the second game of period ping-pong.

\begin{lem}
\label{lem_pp2}
Let $\pi \in \Irr(G_2)$ and $\tau \in \Irr(G')$ be tempered representations such that
\[
\Hom_{G_2 \times G'}(\Pi, \pi \boxtimes \tau) \neq 0.
\]
Then we have the following sequence of natural inclusions:
\begin{align*}
\Hom_{N_2}(\pi, \psi_E) \stackrel{(1)}{\subseteq} \Hom_{N_2}(\Theta(\tau), \psi_E) &\stackrel{(2)}{\cong} \bigoplus_C\Hom_{G_{E,C}'}(\tau^\vee, 1)\\
&\stackrel{(3)}{\subseteq} \bigoplus_C\Hom_{G_{E,C}'}(\Theta(\pi^\vee), 1) \stackrel{(4)}{\cong} \Hom_{G_2}(\bigoplus_C R_C(E), \pi^\vee).
\end{align*}
Here the sum is taken over all $C$ such that $E \oplus C = J$.
If any one of these spaces is finite-dimensional, then inclusions (1) and (3) are in fact isomorphisms.
\end{lem}
\begin{proof}
This is analogous to Lemma 6.4 in \cite{gan2021howe}. First, (1) follows from $\Theta(\pi) \twoheadrightarrow \tau$. Next, (2) follows from
\[
 \Hom_{N_2}(\Theta(\tau), \psi_E)  \cong \Hom_{G'}(\Pi_{N_2,\psi_E}, \tau)
\]
combined with Lemma \ref{lem_GrS} and Frobenius reciprocity. The inclusion (3) is a consequence of $\Theta(\pi^\vee) \twoheadrightarrow \tau^\vee$; this follows from the fact that $\Theta(\overline{\pi})$ is the complex conjugate of $\Theta(\pi)$, combined with $\overline{\pi} \cong \pi^\vee$ and $\overline{\tau} \cong \tau^\vee$. Finally, (4) is the see-saw identity.

Now if any one of the above spaces is finite-dimensional, it follows that $\Hom_{N_2}(\pi, \psi_E)$ is finite-dimensional as well. By Proposition \ref{prop_deg_ps_Spin}, we then have
\[
\dim \Hom_{G_2}(\bigoplus_C R_C(E), \pi^\vee) \leq \dim \Hom_{G_2}(I_E(1/2,\omega_{K/F}), \pi^\vee) = \dim \Hom_{N_2}(\pi, \psi_E).
\]
The result follows.
\end{proof}

The following result will allow us to use the above lemma when analyzing non-generic representations:

\begin{lem}
\label{lem_etale}
Let $\pi$ be an irreducible non-generic infinite-dimensional representation of $G_2$. Then there exists an \'{e}tale cubic algebra $E$ such that $\pi_{{N_2},\psi_E}$ is non-zero. Moreover, $\pi_{{N_2},\psi_E}$ is finite-dimensional for any $E$.
\end{lem}

\begin{proof}
This is Lemma 3.4 in \cite{gan2021howe}
\end{proof}

The two games of period ping-pong allow us to conclude the proof of Theorem \ref{thm_Howe}. First, we prove the finiteness of $\Theta(\pi)_{\text{c}}$ (cf.~Proposition 6.7 in \cite{gan2021howe}).

\begin{prop}
\label{prop_no2tempered_nongen}
Let $\pi \in \Irr(G)$ be tempered and non-generic. Then $\Theta(\pi)$ cannot have two tempered irreducible quotients. In particular, $\Theta(\pi)_{\text{c}}$ is irreducible or $0$. 
\end{prop}

\begin{proof}
Let $\tau_1, \tau_2 \in \Irr(G')$ be irreducible and tempered; assume that $\Theta(\pi) \twoheadrightarrow \tau_1 \oplus \tau_2$. Since $\pi$ is non-generic, there is an \'{e}tale cubic algebra $E$ such that $d \defeq \dim \Hom_{N_2}(\pi^\vee,\psi_E)$ is finite-dimensional and non-zero. Lemma \ref{lem_pp2} (applied first to $\pi^\vee, \tau_1^\vee$, and then to $\pi^\vee$, $\tau_2^\vee$) now shows
\[
d = \dim  \bigoplus_C\Hom_{G_{E,C}'}(\tau_1, 1) =\dim  \bigoplus_C\Hom_{G_{E,C}'}(\Theta(\pi), 1) =  \dim  \bigoplus_C\Hom_{G_{E,C}'}(\tau_2, 1).
\]
However, this is impossible, since it would imply 
\[
d = \dim  \bigoplus_C\Hom_{G_{E,C}'}(\Theta(\pi), 1) \geq \dim  \bigoplus_C\Hom_{G_{E,C}'}(\tau_1, 1) +  \dim  \bigoplus_C\Hom_{G_{E,C}'}(\tau_2, 1) = 2d.
\]
Therefore, we have arrived at a contradiction, and the Proposition is proved.
\end{proof}

We will also need the following result:

\begin{prop}
\label{prop_nongeneric_quotient_implies_unique}
Let $\tau \in \Irr(G')$ be tempered. Let $\pi \in \Irr(G)$ be a tempered, non-generic quotient of $\Theta(\tau)$. Then $\pi$ is the unique 	irreducible tempered quotient of $\Theta(\tau)$.
\end{prop}

\begin{proof}
Since $\pi$ is non-generic, Lemma \ref{lem_etale} shows that there is a cubic algebra $E$ such that $d \defeq \dim \Hom_{N_2}(\pi,\psi_E)$ is finite-dimensional and non-zero. Now Lemma \ref{lem_pp2} shows that $\dim \Hom_{N_2}(\Theta(\tau),\psi_E) = d$. If $\pi'$ is another tempered irreducible quotient of $\Theta(\pi)$, then Lemma \ref{lem_pp2} shows that $\dim \Hom_{N_2}(\pi',\psi_E) = d$. But this implies
\[
d = \dim \Hom_{N_2}(\Theta(\tau),\psi_E) \geq \dim \Hom_{N_2}(\pi,\psi_E) + \dim \Hom_{N_2}(\pi',\psi_E) = 2d,
\]
which is a contradiction. The Proposition is proved.
\end{proof}

Next, we prove that $\Theta(\pi)$, if non-zero, has a unique irreducible quotient.

\begin{prop}
\label{prop_unique_quotient}
Let $\pi \in \Irr(G)$ be tempered such that $\Theta(\pi) \neq 0$. Then $\Theta(\pi)$ has a unique irreducible quotient.
\end{prop}

\begin{proof}
Remark \ref{rem_temp2temp} shows that $\Theta(\pi)$ can only have tempered quotients. The result now follows from Proposition \ref{prop_no2tempered_gen} (when $\pi$ is generic) and Proposition \ref{prop_no2tempered_nongen} (when $\pi$ is non-generic).
\end{proof}

\begin{prop}
\label{prop_injectivity_tempered}
Let $\pi_1, \pi_2 \in \Irr(G)$ be tempered. Then $
0 \neq \theta(\pi_1) \cong \theta(\pi_2)$ implies $\pi_1 \cong \pi_2$.
\end{prop}

\begin{proof}
If $\pi_1$ and $\pi_2$ are both generic, this follows from Proposition \ref{prop_no2_gen_quotients} applied to $\tau = \theta(\pi_1) \cong \theta(\pi_2)$. If either is non-generic, then the result follows from Proposition \ref{prop_nongeneric_quotient_implies_unique}.
\end{proof}

\section{Explicit theta correspondences}
In this section we discuss lifts of (non-cuspidal) tempered representations of $G'$.

\subsection{Representations of the unitary group}  

Recall that $I(\chi, s)$ denotes the principal series representation of $\PU_3(K)$ obtained by inducing $\chi\cdot|N_{K/F}|^s$ (with $\chi$ unitary) from $T' \cong K^\times$. We shall denote $I(\chi, 0)$ simply by $I(\chi)$. The unique non-trivial element of the Weyl group conjugates $\chi^{-1}$ to $\chi^{\sigma}$, where $\sigma$ is the non-trivial element of $\mathrm{Gal}(K/F)$. It is easy to argue that the principal series representations $I(\chi, s)$ have reducibility points only if $\chi^{-1}=\chi^{\sigma}$, that is, $\chi$ is conjugate-dual. Then there is a dichotomy at play. If $I(\chi)$ is reducible, then $I(\chi, s)$ is irreducible for all $s\neq 0$. If $I(\chi)$ is irreducible then there exists $s_0>0$ such that $I(\chi, s)$ are irreducible for $s\neq \pm s_0$. The representation $I(\chi, s_0)$ is a standard module of length two, 
with a unique irreducible quotient and a discrete series representation as a unique submodule. By \cite{Go}  $I(\chi)$ is irreducible if and only if
\[ 
L(\chi, s) L(\mathrm{As}^-(\chi), 2s) 
\] 
has a (simple) pole at $s=0$. More precisely, if $L(\chi, s)$ has a pole at $0$, then $\chi=1$ and $s_0=1$; if $L(\mathrm{As}^-(\chi), 2s)$ has a pole at $0$, 
then $\chi$ is conjugate-symplectic and $s_0=1/2$.  Now the following summarizes our discussion:  

\begin{prop} Let $\chi$ be a conjugate-dual character of $K^{\times}$. Then the principal series $I(\chi,s)$, for $s\geq 0$, reduces as follows: 
\begin{enumerate} 
\item If $\chi=1$, the trivial representation is the quotient, and the Steinberg representation $\mathrm{St}$ is a submodule at $s=1$.  
\item If $\chi\neq 1$ is conjugate-orthogonal then reduction occurs at $s=0$, 
\[ 
I(\chi) = I(\chi)_{\mathrm{gen}} \oplus  I(\chi)_{\mathrm{deg}} 
\] 
where $I(\chi)_{\mathrm{gen}} $ is Whittaker generic.   
\item If $\chi$ is conjugate-symplectic then $I(\chi, s)$ reduces at $s=1/2$. The Whittaker generic submodule is a discrete series $\delta(\chi)$ whose Langlands parameter 
\cite[Section 10]{GGP} is a 3-dimensional conjugate-orthogonal  representation 
\[ 
\chi^{-2} \oplus \chi\otimes V_2 
\] 
of $K^{\times} \times \mathrm{SL}_2$, a quotient of the Weil--Deligne group of $K$ by the commutator of $W_K$, 
where $V_2$ is the irreducible two-dimensional representation of $\mathrm{SL}_2$. 
\end{enumerate} 
\end{prop} 

\begin{rem} The Langlands parameter of $\delta(\chi)$ should perhaps be expressed using $\chi^{-1}=\chi^{\sigma}$ instead of $\chi$. 
 However, the theta lift of both $\delta(\chi)$ and $\delta(\chi^{\sigma})$ is the same representation of $G_2$, so this imprecision is harmless.  
\end{rem}

We may now describe the theta correspondence for (limits of) discrete series of $\mathrm{PU}_3(K)$ discussed above, and for extensions of those representations to 
$\mathrm{PU}_3(K)\rtimes \mathrm{Gal}(K/F)$, when $\mathrm{Gal}(K/F)$-invariant. These lifts will be computed using Jacquet functors and the following additional inputs that are, roughly speaking, 
\begin{itemize} 
\item The correspondence is one-to-one. 
\item It preserves tempered representations. 
\item It preserves generic representations. 
\end{itemize}

Recall that there are two ways to extend $I(\chi)$ to $\mathrm{PU}_3(K)\rtimes \mathrm{Gal}(K/F)$ when $\chi$ is Galois-invariant: $I(\chi^+)$ and $I(\chi^-)$. Here $\chi^+$ is the extension of $\chi$ which appears in the quadratic base change (see \S \ref{subs_oddsnends}). Another way to characterize $I(\chi^+)$ is via Whittaker functionals: $\mathrm{Gal}(K/F)$ acts trivially on the one-dimensional space of Whittaker functionals. If $\tau$ is a constituent of $I(\chi)$, let $\tau^+$ (resp.~$\tau^-$) be the extension of $\tau$ contained in $I(\chi^+)$ (resp.~$I(\chi^-)$). Using Jacquet functors like in Proposition \ref{prop_nontemp_lifts_1}, it follows that theta lifts of constituents of $I(\chi^-)$ are trivial unless the constituent is $\mathrm{St}^-$.

In order to state the result, let $\chi$ be a conjugate-dual character of $K^{\times}$ and let $\nu(\chi)$ be the irreducible representation of $\mathrm{GL}_2(F)$ corresponding to 
the two-dimensional representation $\rho(\chi)$ of the Weil group $W_F$. 
Let $I_{Q_1}^G(\nu(\chi), s)$ denote the principal series where we induce $\nu(\chi)$ twisted by $|\mathrm{det}|^{s}$. 
 If $\chi$ is conjugate-orthogonal, $\chi\neq 1$, then the central character of $\nu(\chi)$ is $\omega_{K/F}$ and $I_{Q_1}^G(\nu(\chi))$ is reducible,
 \[ 
 I_{Q_1}^G(\nu(\chi))= I_{Q_1}^G(\nu(\chi))_{\mathrm{gen}}\oplus I_{Q_1}^G(\nu(\chi))_{\mathrm{deg}}. 
 \]  
 If $\chi$ is conjugate-symplectic, then the central character of 
$\nu(\chi)$ is trivial and  $I_{Q_1}^G(\nu(\chi), 1/2)$ has a Whittaker generic tempered submodule. 
 If $\chi$ is not $\mathrm{Gal}(K/F)$-invariant, then $\nu(\chi)$ is a cuspidal representation and the tempered submodule is a discrete series representation.   
 If $\chi$ is $\mathrm{Gal}(K/F)$-invariant, then  
$\rho(\chi)= \chi_1 \oplus \chi_2$ for a pair of mutually inverse characters of $F^{\times}$, and the tempered submodule is 
\[ 
I_{Q_2}^G( \mathrm{st}_{\chi_1}) \cong I_{Q_2}^G( \mathrm{st}_{\chi_2}),
\]  where $\mathrm{st}_{\chi_i}$ denotes a twist of the Steinberg representation of $\mathrm{GL}_2(F)$ by the character $\chi_i(\mathrm{det})$. 

Finally, we recall the $A$-packet discussed in Section \ref{subs_B3parabolic}. The packet contains two representations: a supercuspidal, and the Langlands quotient of $i_{Q_1}^{G_2}(\lvert\det\rvert\otimes\tau)$, with $\tau$ equal to the tempered representation $1 \times \omega_{K/F}$. In the following proposition, we consider the corresponding discrete series $L$-packet (again attached to the subregular unipotent orbit and the cubic etal\'e algebra $F+K$). Its elements are obtained by applying the Aubert involutions to the elements of the $A$-packet; in particular, we have a supercuspidal representation, and a generic discrete series representation contained in $i_{Q_1}^{G_2}(\lvert\det\rvert\otimes\tau)$ as a submodule.

\begin{prop}  We have: 

\begin{enumerate} 
\item  $\{\theta(\mathrm{St}^+) , \theta(\mathrm{St}^-)\}$ is the discrete series $L$-packet attached to the subregular unipotent orbit and the cubic etal\'e algebra $F+K$.  $\theta(\mathrm{St}^-)$ is supercuspidal.  
\item If $\chi\neq 1$ is conjugate-orthogonal then 
\[ 
\theta(I(\chi)_{\mathrm{gen}}) = I_{Q_1}^G(\nu(\chi))_{\mathrm{gen}} \text{ and } \theta(I(\chi)_{\mathrm{deg}}) = I_{Q_1}^G(\nu(\chi))_{\mathrm{deg}}.  
\]  
If $\chi$ is $\mathrm{Gal}(K/F)$-invariant, then the statements involve constituents of 
$I(\chi^+)$.  
\item If $\chi$ is conjugate-symplectic then 
$\theta(\delta(\chi))$ is a Whittaker generic, tempered submodule of $I_{Q_1}^G(\nu(\chi), 1/2)$. 
 If $\chi$ is $\mathrm{Gal}(K/F)$-invariant, the lift is of $\delta(\chi)^+$. 
\end{enumerate} 
\end{prop} 

\begin{proof}
\begin{enumerate} 
\item Let $\{\sigma,\pi\}$ be the $L$-packet, with $\sigma$ supercuspidal. We use the description of $r_{U'}(\Pi)$ from Proposition \ref{prop_D4}. First, note that $\Theta(\sigma)$ (a representation of $G'$) is non-trivial, because $\sigma$ appears as a quotient of (T3). Thus $\Theta(\sigma)$ appears as a quotient of the minimal representation $\Pi$. It follows that $\sigma \otimes r_{U'}(\Theta(\sigma))$ is a quotient of $r_{U'}(\Pi)$. Since $\sigma$ is supercuspidal, $\sigma \otimes r_{U'}(\Theta(\sigma))$ cannot appear in (B3); therefore, it is a quotient of (T3). Thus $r_{U'}(\Theta(\sigma))$ is one-dimensional, a quotient of (T3); notice that $r_{U'}(\Theta(\sigma))$ is precisely the exponent of $\text{St}^-$. Moreover, the supercuspidal part of $\Theta(\sigma)$ is necessarily $0$ because of the one-to-one property (Theorem \ref{thm_Howe}).  This shows $\Theta(\sigma) = \text{St}^-$.

It remains to determine $\theta(\text{St}^+)$. We know that $\theta(\text{St}^+)\neq 0$ (because $\text{St}^+$ is generic; see Corollary \ref{cor_genericTheta}) and tempered. This implies, using Proposition \ref{prop_D4} again, that $\theta(\text{St}^+)$ appears as a quotient in $B3$. From here it follows that $\theta(\text{St}^+)$ is a subquotient of $i_{Q_1}^{G_2}(\lvert\det\rvert^{-1}\otimes (1 \times \omega_{K/F}))$. Since $\theta(\text{St}^+)$ is tempered, the claim follows.

\item Again, we look at the description of the Jacquet module in Proposition \ref{prop_D4}. The two constituents of $I(\chi)$ are not distinguished by Jacquet modules: we have $r_{U'}(I(\chi)_{\text{gen}}) = r_{U'}(I(\chi)_{\text{deg}}) = \chi$. Note that $\chi$ appears as a quotient in (B3). As explained in \ref{subs_oddsnends}, the similitude correspondence arising from the representation $\tilde{\omega}$ in (B3) now shows that $\nu(\chi) \otimes \chi$ is a quotient of $\tilde{\omega}$. Inducing, we get that $I_{Q_1}^{G_2}(\nu(\chi))$ is a quotient of $r_U'(\Pi)$; in other words (using Frobenius reciprocity), $I_{Q_1}^{G_2}(\nu(\chi))$ is a quotient of $\Pi$. In particular, both $I_{Q_1}^{G_2}(\nu(\chi))_{\text{gen}}$ and $I_{Q_1}^{G_2}(\nu(\chi))_{\text{gen}}$ have non-zero theta lifts which are constituents of $I(\chi)$. Since lifts of generic representations remain generic (Corollary \ref{cor_genericTheta}), we must have $\theta(I(\chi)_{\mathrm{gen}}) = I_{Q_1}^{G_2}(\nu(\chi))_{\text{gen}}$; by the one-to-one property it now follows that $\theta(I(\chi)_{\mathrm{deg}}) = I_{Q_1}^{G_2}(\nu(\chi))_{\text{deg}}$.

\item The proof here is the same as for $\theta(\text{St}^+)$ in case (i).
\end{enumerate}
\end{proof}

\section{Mini theta} 
\label{subs_mini}
It is possible that cuspidal representations of $G'$ lift to non-cuspidal representations of $G$. From Jacquet functors, it is clear that such representations of $G'$ 
are lifts from the Levi $L_2\cong \GL_2(F)$ via the minimal representation of $M_2$.  
We shall describe this mini-theta correspondence by relating it to a classical theta correspondence for unitary groups.

\smallskip 

Let $M_2^{\circ}$ be the connected component of $M_2$. Take a generator of the group of algebraic characters of $M_2^{\circ}$, and let $M_2^{\mathrm{det}}$ be the index two subgroup of $M_2^{\circ}$ of elements such that the character takes values in $N_{K/F}(K^{\times})$. Furthermore, the restriction of the character to $L_2\cong \GL_2(F)$ is the determinant (or its inverse) and we can define $L_2^{\mathrm{det}}\cong \GL_2(F)^{\mathrm{det}}$ analogously.  Thus we have a group isomorphic to 
$\mathrm{PU}(3) \times \GL_2(F)^{\mathrm{det}}$ contained in $M_2^{\mathrm{det}}$. We shall now describe $M_2^{\mathrm{det}}$ and this embedding explicitly. To that end we need to discuss unitary groups briefly. 

\smallskip 
Let $\mathrm{GU}(n)$ denote the group of similitudes of an $n$-dimensional Hermitian space, and let $\mathrm{GU}(n)^{\mathrm{det}}$ denote the index 
two subgroup of elements such that the similitude takes value in $N_{K/F}(K^{\times})$.   
The forms of $\GU(2)$ can be described using quaternion algebras. Let $B$ be a quaternion algebra and fix an 
embedding of $K$ into $B$. The right multiplication by $K$ turns $B$ into a 2-dimensional symmetric Hermitian space, the Hermitian form given by the quaternion norm. 
We have 
\[ 
\GU(2) \cong (B^{\times} \times K^{\times} ) / \Delta F^{\times}
\] 
where $B^{\times}$ acts on $B$ from the left, and $K^{\times}$ from the right, by inverse. 
The center of this group is $(F^{\times} \times  K^{\times})/\Delta F^{\times}\cong K^{\times}$. The subgroup $\GU(2)^{\mathrm{det}}$ consists of pairs 
$(g,z)$ such that the norm of $g$ is in $N_{K/F}(K^{\times})$.  

\smallskip 
Assume now that $B$ is split, so that $B^{\times}= \GL_2(F)$. Let  $\mathrm{U}(6)$ be the unitary group corresponding to the Hermitian space $B\oplus B \oplus B$.  Then 
by Section 7.2 in \cite{gan2021twisted}
\[ 
M_2^{\mathrm{det}}\cong \{ (z,g) \in (K^{\times} \times \mathrm{U}(6))/\Delta \mathrm{U}(1) ~|~  z/\bar z = \mathrm{det} g \} 
\] 
where $\mathrm{U}(1)$ is embedded into $K^{\times} \times \mathrm{U}(6)$ by $z\mapsto (z^3,z)$ and  $\bar z$ denotes the action of the non-trivial  element in $\mathrm{Gal}(K/F)$.
Let $\mathrm{U}(3)\times \mathrm{U}(2)$ be a dual pair in $\mathrm{U}(6)$ so that $\mathrm{U}(2)$ acts diagonally on the three copies of $B$. We map $\mathrm{U}(3)$ into 
$M_2^{\mathrm{det}}$ by 
\[ 
g\mapsto (\mathrm{det}(g), g).
\] 
 It is easy to check that this map is well defined and trivial on the center $\mathrm{U}(1)$.  Now let 
$g\in \GL_2(F)^{\mathrm{det}}$. Let $z\in K^{\times}$ such that $N_{K/F}(z)=\mathrm{det}(g)$.  Then $(g,z)$ defines an element in $\mathrm{U}(2)$, using the above description 
of $\GU(2)$, and 
\[ 
g \mapsto (z^{-3}, (g,z)) 
\] 
is a well defined map from $\GL_2^{\mathrm{det}}$ into $M_2^{\mathrm{det}}$.  

\smallskip 

The mini-theta correspondence is related to the similitude theta
correspondence for $\GU(2)^{\mathrm{det}}  \times  \GU(3)$, which fits in the seesaw

 \begin{picture}(100,82)(-130,10) 

\put(-6,24){$ \GU(1) \cong K^{\times}  $} 

\put(75,24){$\GU(2)^{\mathrm{det}}$}

\put(37,36){\line(1,1){35}}

\put(10,74){$\GU(3)$}
\put(37,71){\line(1,-1){35}}

\put(75,74){$\GU(6)^{\mathrm{det}}$}

\end{picture}
\vskip 15pt

One uses the splitting character 1 on $\GU(3)$ and $\GU(1)$ on the left hand side,  $\mu$ on
$\GU(6)^{\mathrm{det}}$  and  $\mu^3$ on $\GU(2)^{\mathrm{det}}$ \cite[Section 7]{GGP2}. We start by considering this, without referring to the mini-theta yet.   

In this see-saw, one starts with the trivial rep of $\GU(1)$ on the bottom left, and
one takes an irreducible representation $\tau$ of $\GU(2)^{\mathrm{det}}$  on the bottom right.
Then the see-saw identity is
\[
\Hom_{\GU(2)^{\mathrm{det}}}(  \Theta(1), \tau)  = \Hom_{\GU(1)}( \Theta(\tau), 1) .
\] 
Hence, $\Theta(\tau)$ is a representation of $\GU(3)/ \GU(1)$ i.e.~a representation of $\GU(3)$ with trivial central character. 
 Moreover, with the choice of splitting characters as above, the theta correspondence carries
representations of $\GU(3)$ with trivial central character to representations of $\GU(2)^{\mathrm{det}}$ with central
character $\mu^3$. One can describe this theta correspondence, as a lifting from $\GU(2)^{\mathrm{det}}$ to $\U(3)/\U(1)\cong \GU(3)/ \GU(1) $ as follows. 

\smallskip

An irreducible representation of $\GU(2) \cong (\GL_2(F) \times K^{\times} ) / \Delta F^{\times}$ is of the form $\sigma \boxtimes \chi$ for $\sigma$ an irreducible 
representation of $\GL_2(F)$ and $\chi$ a character of $K^{\times}$, so that the central character of $\sigma$ is $\chi|_{F^{\times}}$. 
Since we are interested only in those irreducible representations of $\GU(2)$ with  the central character $\mu^3$, 
we must take $\chi = \mu^3$. In other words, we are looking at irreducible representations $\sigma$ of
$\GL_2(F)$ whose central character is $\omega_{K/F}$ (since $\mu|_{F^{\times}} = \omega_{K/F}$).

\smallskip 

Note that: 

\begin{itemize} 
\item  the contragredient of such a $\sigma$ is $\sigma^{\vee} \cong \sigma \otimes\omega_{K/F}$, so $\sigma$ is dihedral with respect to $K/F$ if and only if $\sigma$ is self-dual.

\item  hence, the restriction of  $\sigma$ to $\GL_2(F)^{\mathrm{det}}$ is irreducible if and only if $\sigma$ 
is not self-dual; if $\sigma$ is self-dual, the restriction breaks into 2 pieces. 

\end{itemize} 


With $\sigma \boxtimes \mu^3$ given, it gives an $L$-packet of 
\[ 
\U(2) = \{ (g,z): \det(g)= N(z) \}/ \Delta F^{\times} \subset \GU(2).
\] 
If $\phi_{\sigma}$ denotes the $L$-parameter of $\sigma$ (as a $\GL_2(F)$-representation), then the 
$L$-parameter of this $\U(2)$ $L$-packet is the conjugate-symplectic representation of the Weil group $W_K$ 
\[ 
\phi_{\sigma} |_{W_K}    \otimes \mu^3. 
\]

Assume now that $\sigma$ is a discrete series representation of $\GL_2(F)$. Then  $\phi_{\sigma}|_{W_K}$ 
is irreducible if and only if $\sigma$ is non-dihedral with respect to $K/F$, i.e. $\sigma$ is not self-dual. 
In any case, by theta lifting from $\GU(2)^{\mathrm{det}}$ to $\U(3)/\U(1)$, the representations of $\U(3)$ we get has
the $L$-parameter
\[
(*)  \hskip 15pt   \phi_{\sigma}|_{W_K}   +  1,
\] 
see Theorem A in \cite{GeRoSo}. Observe that this is a conjugate-orthogonal representation of $W_K$ of determinant one: 
$\phi_{\sigma}|_{W_K}$ has trivial determinant, since $\mathrm{det}\phi_{\sigma} = \omega_{K/F}$. 

\vskip 15pt 

More precisely:

\begin{itemize} 

\item if $\sigma$ is not  self-dual, i.e. $\sigma$ not dihedral, we are starting with a
singleton $L$-packet on $\GU(2)^{\mathrm{det}}$ and its lift is the unique generic representation $\Sigma$ in the
$L$-packet of $\U(3)$ with $L$-parameter $(*)$. That $L$-packet has another element which is
lifted from the non-quasi-split form of $\GU(2)^{\mathrm{det}}$.   

\item if $\sigma$ is self-dual, then the $L$-packet we started with on $\GU(2)^{\mathrm{det}}$ has 2
elements, distinguished by their Whittaker support, and  their theta lifts are two
representation $\Sigma_{\mathrm{gen}}$ and $\Sigma_{\mathrm{deg}}$ in the $L$-packet with $L$-parameter $(*)$; there are two other representations lifted from 
he non-quasi-split form of $\GU(2)^{\mathrm{det}}$.  

\end{itemize}


 \vskip 10pt 

With the above understanding of the similitude theta lifting in hand, let us  return
to the problem of mini-theta.  The representation  $\mu^{-1} \boxtimes \Theta(1)$ of $K^{\times} \times \mathrm{U}(6)$ descends to 
a representation of $M_2^{\mathrm{det}}$  which now is independent of $\mu$, see Section 8.4 in \cite{gan2021twisted}, and 
\[ 
\Ind_{M_2^{\mathrm{det}}}^{M_2^{\circ}}  \,\mu^{-1} \boxtimes \Theta(1) 
\] 
is the minimal representation of $M_2^{\circ}$. 
From the formulas for the embedding of  $\mathrm{PU}(3)$ and $\GL_2(F)^{\mathrm{det}}$ into $M_2^{\mathrm{det}}$, it is easy to check that 
these two groups act on $\mu^{-1} \boxtimes \Theta(1)$ in the same way as they act in the classical
see-saw pair above. Putting things together,  we have: 

\begin{prop}  Let $\sigma$ be a discrete series representation of $\GL_2(F)$ with the central character $\omega_{K/F}$.   Then:  
\begin{itemize} 

\item  If $\sigma$ is not self-dual, i.e. not dihedral w.r.t. $K/F$, the $L$-packet $(*)$ has one representation $\Sigma$ 
we are considering. Under the mini-theta, it lifts to $\sigma + \sigma^{\vee}$, and under the theta lift to $G_2$, it lifts to $\Ind_P^{G_2} \sigma \cong 
\Ind_P^{G_2}( \sigma^{\vee} )$  

\item 
 If $\sigma$ is self-dual, i.e. dihedral w.r.t. $K/F$, then the $L$-packet $(*)$ has two representations $\Sigma_{\mathrm{gen}}$ and $\Sigma_{\mathrm{deg}}$  we
are considering. Under the mini-theta, they lift to $\sigma$. Under the theta lift to $G_2$,
these two representations of  $\mathrm{PU}(3)$  lift to the 2 constituents of
$\Ind_P^{G_2}(\sigma)$.  

\end{itemize} 
\end{prop} 

\vskip 10pt 

We still need to go from $\mathrm{PU}(3)$ to $\mathrm{PU}(3) \rtimes \mathbb Z/2\mathbb Z$.
The action of $\mathbb Z/2\mathbb Z $ on $\Irr(\mathrm{PU}(3))$ is sending $\pi$ to
$\pi^{\vee}$ (see \cite{MVW}).  The representations in the L-packet $(*)$ are all self-dual, and hence each has two 
extensions to $\mathrm{PU}(3) \rtimes \mathbb Z/2\mathbb Z$. Of course, by the one-to-one result, only one of these extensions can lift to 
a summand of $\Ind_P^{G_2} \sigma $. The other extension should not lift to $G_2$, however, we cannot exclude that it lifts to a cuspidal representation of $G_2$.

 \vskip 15pt 
   
 Finally, let us discuss what happens on the level of $L$-parameters. 
  The $L$-parameter $(*)$ is a 3-dimensional rep $\rho: W_K \rightarrow \SL_3(\mathbb C)$ of the form   
  $\rho = \phi |_{W_K} +1$,  where $\phi : W_F \rightarrow \GL_2(\mathbb C)$ has $\mathrm{det}(\phi) = \omega_{K/F}$.  
Also, $\rho$ is the restriction to $W_K$ of an $L$-parameter 
\[ 
\rho':  W_F \rightarrow \SL(3) \rtimes \mathbb Z/2\mathbb Z 
\] 
where the latter is the $L$-group of $\mathrm{PU}(3)$. 
Using the further inclusion $\SL(3) \rtimes \mathbb Z/2\mathbb Z  \rightarrow  G_2 $ 
the 7-dimensional representation of $G_2(\mathbb C)$, as a $W_F$-module, decomposes as 
 \[ 
(\Ind_{W_K}^{W_F}  \rho ) +   \omega_{K/F}  
= ( \phi\cdot  \omega_{K/F} +  \phi + 1+ \omega_{K/F} )  + \omega_{K/F}.
\] 
Recalling that $\phi \cdot \omega_{K/F} = \phi^{\vee}$, and regrouping (i.e. conjugating),
we rewrite this as:
\[ 
(\phi + \omega_{K/F})  + (\phi^{\vee} + \omega_{K/F})  +  1.
\] 
This parameter factors through the Levi $\GL_2(\mathbb C)$ in $\SL_3(\mathbb C)\subset G_2(\mathbb C)$, therefore it is a parameter of the induced representation 
$\Ind_P^{G_2} \sigma \cong \Ind_P^{G_2}( \sigma^{\vee} )$.

\bibliographystyle{amsplain}
\bibliography{bibliography}

\end{document}